\documentclass[a4paper,reqno]{amsart}
\usepackage{amssymb}
\usepackage{amsmath}
\usepackage{mathtools}
\usepackage{ifpdf}
\usepackage{vmargin}

\usepackage{graphicx}  
\usepackage{subfigure}
\usepackage{epstopdf}
\usepackage{caption}

\usepackage{setspace}
\usepackage{xcolor}

\ifpdf 
\usepackage[pdftex]{hyperref}
\else
\usepackage[hypertex]{hyperref}
\fi

\hypersetup{
pdftitle={Minimal mass blow up solutions for nonlinear Schr\"odinger equation with delta potential},
pdfauthor={Xingdong Tang, Guixiang Xu},
}

\newtheorem{theorem}{Theorem}
\newtheorem{proposition}[theorem]{Proposition}
\newtheorem{lemma}[theorem]{Lemma}
\newtheorem{remark}[theorem]{Remark}
\newtheorem{corollary}[theorem]{Corollary}

\numberwithin{equation}{section}

\DeclareMathOperator{\sech}{sech}
\DeclareMathOperator{\arctanh}{arctanh}

\DeclarePairedDelimiter{\norm}{\lVert}{\rVert}

\newcommand{\psld}[2]{\left( #1,#2 \right)_{2}}

\newcommand{\dual}[2]{\left\langle #1,#2 \right\rangle}

\newcommand{\eps}{\varepsilon}

\newcommand{\N}{\mathbb{N}}
\newcommand{\R}{\mathbb{R}}
\newcommand{\C}{\mathbb C}

\DeclareMathOperator{\ModOp}{Mod_{op}}
\DeclareMathOperator{\Mod}{Mod}

\renewcommand{\Re}{\mathrm{Re}}
\renewcommand{\Im}{\mathrm{Im}}

\newcommand{\bee}{\begin{eqnarray*}}
\newcommand{\eee}{\end{eqnarray*}}

\numberwithin{equation}{section}
\numberwithin{theorem}{section}

\begin{document}
	\onehalfspacing

\title[Minimal mass blow-up solutions for NLS with delta potential]{Minimal mass blow-up solutions for the   $L^2$-critical NLS with the Delta potential for radial data in one dimension}


\author[]{Xingdong Tang}
\address{\hskip-1.15em Xingdong Tang 
	\hfill\newline School of Mathematics and Statistics, \hfill\newline Nanjing Univeristy of Information Science and Technology, 
	\hfill\newline Nanjing, 210044,  People's Republic of China.}
\email{txd@nuist.edu.cn}

\author[]{Guixiang Xu}
\address{\hskip-1.15em Guixiang Xu
	\hfill\newline Laboratory of Mathematics and Complex Systems,
	\hfill\newline Ministry of Education,
	\hfill\newline School of Mathematical Sciences,
	\hfill\newline Beijing Normal University,
	\hfill\newline Beijing, 100875, People's Republic of China.}
\email{guixiang@bnu.edu.cn}



\begin{abstract}
We consider the $L^2$-critical nonlinear Schr\"odinger equation (NLS) with the delta potential
 $$i\partial_tu +\partial^2_x u + \mu \delta u +|u|^{4}u=0, \, \, t\in \R, \, x\in \R , $$
 where $ \mu \in \R$, and $\delta$ is the Dirac delta distribution at $x=0$.   Local well-posedness theory together with sharp Gagliardo-Nirenberg inequality and the conservation laws of mass and energy implies that the solution with mass less than $\|Q\|_{2}$ is global existence in $H^1(\R)$, where $Q$ is the ground state of the $L^2$-critical NLS without the delta potential (i.e. $\mu=0$).  
 
 We are interested in the dynamics of the solution with threshold mass $\|u_0\|_{2}=\|Q\|_{2}$ in $H^1(\R)$. First, for the case $\mu=0$, such blow-up solution exists due to the pseudo-conformal symmetry of the equation,  and is unique up to the symmetries of the equation  in $H^1(\R)$ from \cite{Me93:NLS:mini sol} (see also \cite{HmKe05:NLS:mini blp}), and recently in $L^2(\R)$ from  \cite{Dod:NLS:L2thrh1}. Second, for the case $\mu<0$, simple variational argument with the  conservation laws of mass and energy implies that radial solutions with threshold mass exist globally in $H^1(\R)$. Last, for the case $\mu>0$, we show the existence of radial threshold solutions with blow-up speed determined by the sign (i.e. $\mu>0$) of the delta potential perturbation since the refined blow-up profile to the rescaled equation is stable in a precise sense.  The key ingredients here including the Energy-Morawetz argument and compactness method as well as the modulation analysis are close to the original one in \cite{RaS11:NLS:mini sol}  (see also \cite{KrLR13:HalfW:nondis, LeMR:CNLS:blp, Mart05:Kdv:N sol, MaP17:BO:mini sol, MeRS14:NLS:blp}). 
\end{abstract}

\maketitle
 


\section{Introduction}

In this paper, we consider the  $L^2$-critical nonlinear Schr{\"o}dinger equation with the delta potential
\begin{equation}\label{eq:dnls}
\begin{cases}
i \partial_t u +\partial^2_x u +\mu \delta u+|u|^4u=0, & (t,x)\in\R \times\R, \\
u(0,x)=u_0(x)\in H^1(\R), &
\end{cases}
\end{equation}
where $u$ is a complex-valued function of $(t,x)$,
$\mu \in \R $, $\delta$ is the Dirac delta distribution at the origin.
For $\mu=0$, it is focusing, $L^2$-critical nonlinear Schr\"odinger equation (NLS) in one dimension, here we call it the $L^2$-critical NLS  since the scaling transformation 
$$u_{\lambda}(t,x)=\lambda^{\frac 12}u(\lambda^2 t, \lambda x)$$ 
leaves the $L^2$ norm invariant
\begin{equation*}
\|u_{\lambda}(t,\cdot)\|_{2}=\|u(\lambda^2t, \cdot)\|_{2}.
\end{equation*} 
A series of studies dealt with the $L^2$-critical NLS, we can refer to \cite{Ca03:book, Dod:NLS:sct, Dod19:book, Dod:NLS:L2thrh1, Fi:NLS:book, Me93:NLS:mini sol, SuS99:NLS:book}, and references therein. For $\mu \not = 0$, it appears in various physical models with a point defect on the line, for instance, quantum mechanics  \cite{AlGKH88:Phys:book}, nonlinear optics \cite{GoHW04:dNLS:PhysD, HoZZ07:NLS:scat, HoZ09:dNLS:breat} and references therein. We can refer to \cite{AdN:JPA, AlGKH88:Phys:book} for more details on the $\delta$-perturbation of strength $\mu$ of the self-adjoint operator $\partial^2_x$. The appearance of the delta potential destroys spatial translation, scaling transformation and pseudo-conformal transformation invariances of \eqref{eq:dnls}. For the case $\mu<0$, it corresponds to the repulsive delta potential,  while for the case  $\mu>0$, it corresponds to the attractive delta  potential. 

 Local well-posedness result for \eqref{eq:dnls} is well understood in $H^1(\R)$ by many authors, for example, by Cazenave in \cite[Theorem 3.7.1]{Ca03:book},  Fukuizumi, Ohta and Ozawa in \cite{FuOZ08:DNLS:sta} and Masaki, Murphy and Segata in \cite{MaMS18:dNLS:sct}. More precisely, we have
\begin{proposition}\label{prop:LWP}
	For any $u_0\in H^1(\R)$, there exists a unique maximal lifespan solution
	$u\in \mathcal{C} \left( \left(-T_{*}, T^{*}\right), H^1(\R) \right) \cap \mathcal{C}^1\left( \left(-T_{*}, T^{*}\right), H^{-1}(\R) \right)$ of \eqref{eq:dnls} 
	satisfying the the following blow-up criterion,
	\begin{equation}\label{lwp:blp crit}
    T^{*}<+\infty \text{\;implies that\;}  \displaystyle\lim\limits_{t\nearrow T^{*}}\norm{\partial_x u(t) }_{2}=+\infty.
	\end{equation}
	Moreover, the mass and the energy are conserved
	under the flow \eqref{eq:dnls}, i.e., for any $t\in \left(-T_{*}, T^{*}\right)$, we have
	\begin{gather}
	\label{con:mass}
	M(u(t)):=\frac{1}{2}\int_{\R}|u(t,x)|^2 d x = M(u_0),
	\\
	\label{con:energy}
	E(u(t))
	: =
	\int_{\R}\left[ \frac{1}{2}| \partial_x u (t,x)|^2-\frac{1}{2}\mu \  \delta(x)|u(t,x)|^2-\frac{1}{6} |u(t,x)|^{6}   \right] dx= E(u_0).
	\end{gather}
	
\end{proposition}
 
\subsection{The $L^2$-critical case $\mu=0$.} Let us firstly recall the well-known results of \eqref{eq:dnls} for the case $\mu=0$. 
 From the variational argument in \cite{We83:NLS:SGN}, the positive, radial ground state solution to 
\begin{equation}\label{eq:grd sol1}
-\partial^2_x Q+Q-|Q|^{4}Q=0, \ \ Q\in H^1(\R), \ \ Q>0, \ \ \mbox{$Q$ radial}
\end{equation}
is the extremizer of sharp Gagliardo-Nirenberg inequality 
\begin{equation}
\label{est:SGN}
\|u\|^{6}_{6}\leq C_{\rm GN}\|u\|_{2}^{4}\|\partial_x u\|_{2}^2,
\end{equation}
where $C_{\rm GN}= 3/\norm{Q}^4_2$. Moreover, the ground state $Q$ is unique up to symmetries of \eqref{eq:grd sol1} from \cite{BL83:ell PDE, Kw89:EPDE:uniq}.  Therefore, for any $u\in H^1(\R)$, we have 
\begin{equation}
\label{est:GNenergy}
 E(u)=E_{\rm crit}(u):=\frac12\norm{\partial_x u}_2^2-\frac{1}{6}\norm{u}_{6}^{6}\geq \frac 12\|\partial_x u\|_{2}^2\left(1-\frac{\|u\|^4_{2}}{\|Q\|^4_{2}}\right),
\end{equation}
which together with  \eqref{con:mass}, \eqref{con:energy} and the blow-up criterion \eqref{lwp:blp crit}  implies that the solutions of \eqref{eq:dnls} with mass less than $\|Q\|_{2}$ globally exist in $H^1(\R)$. Furthermore, it scatters to the linear solutions $e^{it\partial^2_x}u^{\pm}_0$ for some $u^{\pm}_0\in L^2(\R)$ in both time directions in $L^2(\R)$, we can refer to \cite{Dod:NLS:sct} for more details.

For the critical mass case $\|u_0\|_{2}=\|Q\|_{2}$. By applying the pseudo-conformal transformation
\begin{equation}
\label{sym:psdconf}
u(t,x) \longmapsto v(t,x)=\frac{1}{|t|^{\frac 12}}\bar{u} \left(\frac 1t,\frac{x}{t}\right)e^{i\frac{|x|^2}{4t}}
\end{equation} to the solitary solution $u(t,x)=Q(x)e^{it}$, we can obtain the explicit blow-up solution with threshold mass
\begin{equation}
\label{sol:psdconf blp}
S(t,x) = \frac{1}{|t|^{\frac 12}}Q\left(\frac {x}{|t|} \right)e^{-i\frac{|x|^2}{4|t|}}e^{\frac {i}{|t|}},\quad
\|S(t)\|_2= \|Q\|_{2}, \quad\text{and}\quad \|\partial_x S(t)\|_{2}\approx \frac 1{|t|} \,\,\text{as}\,\, t\to 0^-.
\end{equation}
From \cite{Me93:NLS:mini sol}, finite time blow-up solutions with threshold mass $\norm{Q}_2$ are completely classified by F.~Merle in $H^1(\R)$  in the following sense $$\|u(t,\cdot)\|_{2}=\|Q\|_{2}\ \ \mbox{and}\ \ T^*<+\infty\ \ \mbox{imply}\ \ u\equiv S$$ up to the symmetries of \eqref{eq:dnls} (see also \cite{HmKe05:NLS:mini blp} for simplified proof).  Recently B.~Dodson pushes this rigidity result forward in \cite{Dod:NLS:L2thrh1, Dod:NLS:L2thrh2}, and obtains the classification result of finite time blow-up solutions of the focusing,  $L^2$-critical NLS with threshold mass in $L^2(\R^d)$, $1\leq d \leq 15$.

For the super-critical mass case $\|u_0\|_{2}>\|Q\|_{2}$, there are lots of important blow-up results of \eqref{eq:dnls} in $H^1(\R)$, we can refer to  \cite{BoW97:NLS:blp, FanMen:NLS:blp, MeR03:NLS:blp, MeR04:NLS:blp, MeR05:NLS:blp, MeR06:NLS:blp, MeRS13:NLS:instab, PlR07:NLS:blp, Ra05:NLS:blp, Ra06:NLS:blp}, 
and references therein.

\subsection{\texorpdfstring{The repulsive potential case $\mu <0$}{The case mu<0}} 
Now we consider the repulsive potential case $\mu<0$. On the one hand, there are no solitary waves of \eqref{eq:dnls} with the subcritical mass $\|u_0\|_{2}<\|Q\|_{2}$ from classical variational argument in \cite{FuOZ08:DNLS:sta, LeFFKS08:dNLS:instab}. On the other hand, by \eqref{con:mass}, \eqref{con:energy} and  sharp Gagliardo-Nirenberg inequality \eqref{est:SGN}, we can obtain similar estimate to those in \eqref{est:GNenergy}
\begin{equation*}
 E(u(t))=\frac12\norm{\partial_x u(t,\cdot)}_2^2-\frac{\mu}{2} |u(t,0)|^2-\frac{1}{6}\norm{u(t,\cdot)}_{6}^{6}\geq \frac 12\|\partial_x u(t,\cdot)\|_{2}^2\left(1-\frac{\|u(t,\cdot)\|^4_{2}}{\|Q\|^4_{2}}\right).
\end{equation*}
 This together with the blow-up criterion \eqref{lwp:blp crit} implies the global existence of the solutions of \eqref{eq:dnls}  with mass less than $\|Q\|_{2}$ in $H^1(\R)$. Because of the repulsive ($\mu<0$) delta potential perturbation and the lack of the pseudo-conformal symmetry \eqref{sym:psdconf}, we have at the threshold mass that

\begin{theorem}
	\label{thm:thresh1}
	Let $\mu<0$, and radial $u_0\in H^1(\R)$ with  $\|u_0\|_{2}=\|Q\|_{2}$, then  the solution of \eqref{eq:dnls} is global and bounded in $H^1(\R)$.
\end{theorem}

The result follows from the standard concentration compactness argument with the  conservation laws of mass and energy, see proof in Appendix \ref{app:a:gwp}. 

In contrast to the subcritical mass,  there are solitary waves $Q_{\omega, \mu}(x)e^{i\omega t}$ of \eqref{eq:dnls} with the super-critical mass $\|Q_{\omega, \mu}\|_{2}>\|Q\|_{2}$ from standard variational argument in \cite{FuOZ08:DNLS:sta, LeFFKS08:dNLS:instab}, at the same time, these solitary waves are orbitally unstable, and even strongly unstable in $H^1(\R)$, we can refer to \cite{LeFFKS08:dNLS:instab}, and references therein.  Therefore the global existence result in Theorem \ref{thm:thresh1} is sharp in $H^1(\R)$, and we also conjecture that for the repulsive delta potential case, the solutions of \eqref{eq:dnls} with $\norm{u_0}_2\leq \norm{Q}_2$ will scatter to the linear solution in both time directions in $L^2(\R)$ due to the result in \cite{Dod:NLS:sct}. For other long-time asymptotic behavior and scattering results of nonlinear Schr\"odinger equation with the repulsive delta potential, we can refer to \cite{BaV16:dNLS:sct, DeP11:dNLS:IST, IkI17:dNLS:sct, MaMS18:dNLS:sct}, and references therein.

\subsection{\texorpdfstring{The attractive potential case $\mu>0$}{The case epsilon=1}} Let us now  turn to the attractive potential case $\mu>0$, it is main part for the rest of the paper.
First, by the Sobolev inequality and  the Young inequality, we have for any $u\in H^1(\R)$ with $\norm{u}_2 < \norm{Q}_2$ that 
\begin{equation*}
\frac{\mu}{2} |u(t,0)|^2 \leq \frac{\mu}{2} \norm{\partial_x u(t,\cdot)}_2\norm{u(t,\cdot)}_2\leq  \frac 16\|\partial_x u\|_{2}^2\left(1-\frac{\|u(t,\cdot)\|^4_{2}}{\|Q\|^4_{2}}\right) + C(\mu, \norm{u(t,\cdot)}_2),
\end{equation*}
this together with sharp Gagliardo-Nirenberg inequality \eqref{est:SGN} implies that for any $u\in H^1(\R)$, we have
\begin{multline}\label{est:energy}
 E(u(t))=\frac12\norm{\partial_x u(t,\cdot)}_2^2-\frac{\mu}{2} |u(t,0)|^2-\frac{1}{6}\norm{u(t,\cdot)}_{6}^{6}\\
 \geq \frac 13\|\partial_x u(t,\cdot)\|_{2}^2\left(1-\frac{\|u(t,\cdot)\|^4_{2}}{\|Q\|^4_{2}}\right) - C(\mu, \norm{u(t,\cdot)}_2),
\end{multline}
which once again implies the controllability of the kinetic energy $\norm{\partial_x u(t)}^2_{2}$ of the solution of \eqref{eq:dnls} by the  mass and energy, and deduces global well-posedness of the solution of \eqref{eq:dnls} with $\|u_0\|_{2}<\|Q\|_{2}$ in $H^1(\R)$ due to \eqref{con:mass}, \eqref{con:energy} and the blow-up criterion \eqref{lwp:blp crit}. However,  in contrast to both the case $\mu=0$ and the  case $\mu<0$, there exist arbitrarily small solitary waves $ Q_{\omega, \mu}(x)e^{i\omega t}$ as follows, 
\begin{proposition}
	\label{prop:smallsoliton}
Let $\mu>0$. For all $M\in (0,\|Q\|_{2})$, there exists $\omega> \mu^2/4$ and a unique positive,  radially symmetric solution of
	$$-\omega Q_{\omega, \mu} + \partial^2_x Q_{\omega, \mu}+ \mu \delta Q_{\omega, \mu}+Q_{\omega, \mu}^{5}=0, \ \ \|Q_{\omega, \mu}\|_{2}=M.$$
\end{proposition}
The proof of Proposition \ref{prop:smallsoliton} follows from classical variational argument in \cite{FuJ08:dNLS, FuOZ08:DNLS:sta}, we give alternative proof by in Appendix \ref{app:b}. 
Since the solitary waves $ Q_{\omega, \mu}(x)e^{i\omega t}$  themselves don't scatter to the linear solution,   the solution of \eqref{eq:dnls} with mass less than $\norm{Q}_2$ don't scatter in $L^2(\R)$ any more in this case. The necessary condition that $\omega>\mu^2/4$ is related to the eigenvalue of the Schr\"odigner operator $\partial^2_x + \mu \delta$ in \cite{AlGKH88:Phys:book, FuOZ08:DNLS:sta}. In fact, these solitary waves can be explicitly described as following (see  \cite{FuOZ08:DNLS:sta, GoHW04:dNLS:PhysD})
\begin{equation}
\label{sol:grd sol2}
{Q}_{\omega, \mu}(x)=
\left[ 3 \omega
\sech^2\left( 2\sqrt{\omega}|x| +\arctanh\left( \frac{\mu}{2\sqrt{\omega}}\right) \right) \right]^{\frac{1}{4}}.
\end{equation}
By the well-known stability theory in \cite{CaLi82:NLS:Stab, FuOZ08:DNLS:sta, GrSS87:NLS:stab},  these solitary waves are orbitally stable in $H^1(\R)$.

 In contrast to the non-existence result of finite time blow-up solutions with threshold mass for the repulsive case $\mu<0$  in Theorem \ref{thm:thresh1}, the main goal of this paper is to show the existence of finite time blow-up solutions of \eqref{eq:dnls} with threshold mass in the attractive case $\mu>0$ as follows,

\begin{theorem}
	\label{thm:thresh2}
  Let $\mu>0$ and $E_0\in \R$, there exist $t_0<0$ and a radial data $u(t_0)\in H^1(\R)$ with $$\|u(t_0)\|_{2}=\|Q\|_{2}, \ \ E(u(t_0))=E_0,$$ such that the corresponding solution $u(t)$ of \eqref{eq:dnls} blows up at time $T^*=0$ with the blow-up speed:
	\begin{equation}
	\label{blp:rate2}
	\|\partial_x u(t)\|_{2} \approx \frac{1}{|t|^{2/3}}, \quad \text{as}\quad t \to 0^{-}. 
	\end{equation}
\end{theorem}

\begin{remark} We give some remarks on this result.
\begin{enumerate}
\item Energy. Minimal mass blow-up solutions in Theorem \ref{thm:thresh2} have arbitrary energy $E_0\in \R$, which is different with the positive energy $E_{\rm crit}(S(t))$of  the explicit pseudo-conformal blow-up solution $S(t,x)$  in \eqref{sol:psdconf blp}.
	
\item Blow-up speed. The blow-up speed $|t|^{-2/3}$ in \eqref{blp:rate2} is different with the pseudo-conformal blow-up speed $|t|^{-1}$ in \eqref{sol:psdconf blp},  the self similar blow-up speed ${|t|^{-\frac 12}}$ and the log-log blow-up speed for $L^2$-critical NLS in \cite{BoW97:NLS:blp, MeR03:NLS:blp, MeR04:NLS:blp, MeR05:NLS:blp, MeR06:NLS:blp, Pe:NLS:blp}.  The possible blow-up speed for the critical problem is an interesting problem, we can refer \cite{MaMR14:Kdv:clas, MaMR15:Kdv:mini sol, MaMR15:Kdv:exot} for the $L^2$-critical KdV equation, \cite{KrST09:NLW:blp, KrS14:NLW:blp} for the $\dot H^1$-critical wave equation and references therein. 
	
\item  Profile. The analysis provides the following profile of minimal mass blow-up solution
	\begin{equation}\label{sol:blp form}
	u(t,x)=\frac{1}{\lambda^{\frac{1}{2}}(t)}
	Q\left(\frac{x}{\lambda (t)}\right)e^{-i\frac{2}{3} \frac{|x|^2}{4t}}e^{i\gamma(t)}
	+v\left(t,{x}\right)
	\end{equation}
	where $Q$ is the ground state solution of $L^2$-critical NLS, and
	$$\lim_{t\rightarrow 0}\norm{v(t)}_{2}=0,\quad \text{and} \quad
	\lambda(t)\approx C { |t|^{2/3}} , \quad \text{as} \quad t \to 0^-.
	$$
See \eqref{est:small pars} and \eqref{est:small pars2}. 

\item Uniqueness. The uniqueness of minimal mass blow-up solution is an  important problem, which is closely related to classify the compact elements of the flow in the Kenig-Merle's concentration-compactness-rigidity argument \cite{KeM06:NLS:sct}, we can also refer to \cite{Dod:NLS:L2thrh1, Dod:NLS:L2thrh2, DuM09:NLS:thresh sol, DuR10:NLS:thresh sol, KiTV09:NLS:2d rad, LiZ09:NLS:thresh sol, MiWX15:Hart:thresh sol, TaVZ08:NLS:crit element}, and references therein. 
\end{enumerate}
\end{remark}

The results in Theorem \ref{thm:thresh1} and Theorem \ref{thm:thresh2} show the completely different consequence determined by the delta potentials in the existence of minimal mass blow-up solutions of \eqref{eq:dnls} in $H^1(\R)$.   Due to the refined blow-up profile to the rescaled equation is stable in a very precise sense, the construction proof of Theorem \ref{thm:thresh2} using the Energy-Morawetz argument and compactness method as well as the modulation analysis is close to the original non perturbative argument in \cite{RaS11:NLS:mini sol} (see also \cite{KrLR13:HalfW:nondis, LeMR:CNLS:blp, Mart05:Kdv:N sol, MaP17:BO:mini sol}), which is different with the perturbation argument in \cite{BaCD11:NLS:blp}. More precisely,  we adapts the compactness argument under the uniform estimates for the special solution on sufficiently far away time rescaled interval, which in fact can be satisfied by modulation analysis, the Energy-Morawetz estimate of the remainder term $\varepsilon(t,x)$ and the bootstrap argument. The related application of the Energy-Morawetz estimate in the blow-up dynamics can also be found in \cite{MeRS14:NLS:blp, RaS11:NLS:mini sol}. The Energy-Morawetz estimate is also successfully applied by B.~Dodson in \cite{Dod:NLW:sct1, Dod:NLW:sct2} to obtain the global well-posedness and scattering result for the radial solution of the defocusing, nonlinear wave equation in the critical Sobolev space $\dot H^s(\R^3)$ for $\frac12\leq s<1$.

For other results of nonlinear Schr\"odinger equation with the attractive delta potential in $H^1(\R)$, we can  refer to  \cite{FuOZ08:DNLS:sta, GoHW04:dNLS:PhysD, LeFFKS08:dNLS:instab, MaMS20:dNLS:stab, TaX:dNLS:unstab} and references therein for the stability analysis of the solitary waves.

\subsection{Notation}\label{sect:notation} Let us collect the notation and some well-known facts used in this paper. Throughout the paper, we  use the notation $X \lesssim Y$, or $Y \gtrsim X$ to denote the statement that $X\leq C Y$  for some constant $C$, which may vary from line to line.  We use $X=O(Y)$ synonymously with $|X| \lesssim Y$.  We use $X\approx Y$ to denote the statement $X\lesssim Y \lesssim X$. 

Since the appearance of the Dirac delta potential destroys spatial translation invariance of \eqref{eq:dnls} for $\mu > 0$, we only deal with the radial case. The $L^2$ scalar product and $L^r$ norm ($r\geq 1$) are  denoted by 
\begin{equation*}
\psld{u}{v}=\Re\left(\int_{\R}u(x)\bar v(x)\, dx\right),\quad \|u\|_{r} = \left(\int_{\R} |u|^r\, dx\right)^{\frac 1r}.
\end{equation*}
We denote the radial functions in $H^1(\R)$ by $H^1_{\rm rad}(\R)$.

We fix the notation: for $\mu>0$, $u\in \C$,  denote
\begin{equation*}
f(u)=|u|^4 u;\quad 
g(u)=\mu\, \delta u;\quad
F(u)=\frac{1}{6}|u|^{6};\quad
G(u)=\frac{1}{2} \mu\, \delta |u|^{2}.
\end{equation*}
where $\delta=\delta(x)$ is the Dirac delta distribution at the origin and obeys the following scaling property by simple distribution calculation (see \cite{Gr14:CFA}): 
\begin{equation}\label{delta:scal}
\forall x\in \R,\ \lambda>0, \quad \delta(\lambda x)= \frac{1}{\lambda} \delta(x).
\end{equation}

Identifying $\mathbb C$ with $\R^2$, we denote  the Fr\'echet-derivative of functions $f, g, F, G$ by $df$, $dg$, $dF$ and $dG$.  Let  $\Lambda$ be the infinitesimal generator of the $L^2$-scaling transformation, i.e. 
\begin{equation*}
\Lambda=\frac{1}{2}+y\cdot\partial_y.
\end{equation*}
Without loss of generality, we may assume that $\mu^2/4< \omega=1$ in the rest of the paper, the linearized operator around $Qe^{it}$ is 
\begin{equation*}
L_+:=-\partial^2_y+1-5Q^{4},\qquad
L_-:=-\partial^2_y+1-Q^{4}.
\end{equation*}
and the generalized kernel of $$\begin{pmatrix}0 & L_-\\-L_+&0\end{pmatrix}$$ is non-degenerate and spanned by 
 the symmetries of the equation (see \cite{Kw89:EPDE:uniq, We85:NLS:Mod} and \cite{ChGuNaTs07:NLS:spt}). It is described in $H^1_{\rm rad}(\R)$ by  the algebraic relations (we define $\rho$ as the unique radial solution to $L_+\rho =|y|^2Q$)
\begin{equation}\label{fact:algebra}
L_-Q=0,\quad L_+\Lambda Q=-2Q,\quad L_-|y|^2Q=-4\Lambda Q,\quad L_+\rho =|y|^2Q.
\end{equation}
From these algebraic relations, we have
\begin{equation}\label{est:basic fact2}
\psld{Q}{\rho}=-\frac12 \psld{L_+ \Lambda Q}{\rho} =-\frac12 \psld{ \Lambda Q}{L_+ \rho} =-\frac12 \psld{ \Lambda Q}{|y|^2Q} = \frac12 \norm{yQ}^2_2.
\end{equation}

Denote by $\mathcal Y$ the set of radially symmetric functions $f\in \mathcal{C}^\infty(\R\backslash \{0\})$ such that
$$
\forall \alpha\in\N,\quad \exists C_\alpha,\ \kappa_\alpha>0, \ \forall x\in\R\backslash\{0\},\quad 
|\partial^\alpha f(x)|\leq C_\alpha (1+|x|)^{\kappa_\alpha} Q(x).
$$
It follows from the kernel properties of $L_+$ and $L_-$, and \cite[Appendix A]{CoLe11:NLS:slt} or proof of Lemma 3.2 in \cite{MeRS14:NLS:blp} for related arguments) that
\begin{align}
&\forall g\in \mathcal{Y}, \ \exists f_+\in \mathcal{Y}, \ L_+ f_+ = g,\label{eq:Lp}\\
&\forall g\in \mathcal{Y}, \ \psld{g}{Q}=0, \ \exists f_-\in \mathcal{Y}, \ L_- f_- = g.\label{eq:Ln}
\end{align}

For the sake of the localization argument in Section \ref{Sect3: uniform ests}, We introduce the localized function $\phi$ and its scaling. Let $\phi:\R\to \R$ be a smooth even and convex function, nondecreasing on $\R^+$,
such that   
\begin{equation*}
\phi(r)=
\left\{
\begin{aligned}
&\frac12r^2& \text{ for }&r<1,\\
&3r+e^{-r}& \text{ for }&r>2,
\end{aligned}
\right.
\end{equation*}
and set $\phi(x)=\phi(|x|)$.
For  $A\gg1$,  define $\phi_A$ by $\phi_A(x)=A^2\phi\left(\frac{x}{A}\right)$.

This paper is organized as follows. In Section \ref{Sect2: blp profile}, we construct the refined blow-up profile by  according to the Dirac potential perturbation, and obtain the approximate blow-up law affected by the Dirac potential perturbation. 
In Section \ref{Sect3: uniform ests},  we obtain the uniform backwards estimates of the remainder $\varepsilon$ and modulation parameters $ \lambda, b$ on the rescaled time interval by combining the modulation analysis,  the Energy-Morawetz argument with the bootstrap argument on the sufficiently far away rescaled time interval. In Section \ref{Sect4: Exist}, we can construct minimal mass blow-up solution by combining the compactness argument  with the uniform backwards estimates in Section \ref{Sect3: uniform ests}. 
In Appendix \ref{app:a:gwp}, we use the variational argument, and the conservation laws of mass and energy to show the global existence of radial solution of \eqref{eq:dnls} with threshold mass for the repulsive potential case $\mu<0$ in $H^1(\R)$.
In Appendix \ref{app:b}, we use the variational argument to show Proposition \ref{prop:smallsoliton}, that is, the existence of arbitrarily small radial soliton solutions of \eqref{eq:dnls} for the attractive potential case $\mu>0$.

\noindent \subsection*{Acknowledgements.}
G. Xu  was supported by National Key Research and Development Program of China (No. 2020YFA0712900) and by NSFC (No. 11831004). X. Tang was supported by NSFC (No. 12001284).   

\section{Construction of the refined blow-up profile}\label{Sect2: blp profile}

In this section, we construct the refined blow-up profile according to the Dirac delta potential perturbation. 

\subsection{Refined blow-up profile}
We start with the heuristic argument justifying the construction as that in \cite{RaS11:NLS:mini sol} (see also \cite{KrLR13:HalfW:nondis} \cite{LeMR:CNLS:blp} \cite{MaP17:BO:mini sol}). Since the Dirac delta potential term $\mu \delta u$ destroys the invariances of spatial translation and Galilean transformation of \eqref{eq:dnls}, we try to look for a solution with the structure
\begin{equation}\label{blp:form}
u(t,x)=\frac{1}{\lambda^{\frac{1}{2}}(s)}w(s,y)e^{i\gamma(s)-i\frac{b(s)|y|^2}{4}},\qquad
\frac{ds}{dt}=\frac{1}{\lambda^2},\qquad 
y=\frac{x}{\lambda(s)},
\end{equation}
where  $(s,y)$ are the rescaled variables, and  the parameters $\lambda>0, b, \gamma$ is to be determined later. By inserting \eqref{blp:form} into \eqref{eq:dnls},  the profile function $w$, and  $\lambda$, $b$ and $\gamma$ should satisfy the following rescaled equation
\begin{multline}\label{eq:w}
iw_s+\partial^2_y w-w+f(w)+\lambda g(w)  \\-i\left(\frac{\lambda_s}{\lambda}+ b\right)\Lambda w
+(1-\gamma_s)w
+(b_s+b^2)\frac{|y|^2}{4}w-b\left(b+\frac{\lambda_s}{\lambda}\right)\frac{|y|^2}{2}w
=0.
\end{multline}
where we used the fact that \eqref{delta:scal}.
Since we look for the blow-up solution, the parameter $\lambda(s)$ should converge to zero as $s\to \infty$.
Therefore, we expect that the parameters $\lambda$, $b$, $\gamma$ satisfy the modulation equations
\begin{equation}\label{ode:mod}
\frac{\lambda_s}{\lambda}+ b=0,\quad b_s+b^2=0,\quad 1-\gamma_s=0,
\end{equation}
and that
\begin{equation}\label{w:app sol1}
w(s,y)=Q(y)
\end{equation}
is an approximate solution of \eqref{eq:w} with zero order term of $\lambda$ and $b$. However, the first order error term $\lambda g(Q)$ cannot be regarded as small perturbation in minimal blow-up analysis due to the slow variation property of the parameter $\lambda$ (In fact $\lambda(s) \approx s^{-2}$ for sufficiently large $s$, see Lemma \ref{lem:sol:ode app} and \eqref{est:small est:s} in Proposition \ref{prop:unif est}). We rewrite \eqref{eq:w} as follows
\begin{multline}\label{eq:w2}
iw_s+\partial^2_y w-w+f(w)+\lambda g(w) +\theta \frac{|y|^2}{4}w \\-i\left(\frac{\lambda_s}{\lambda}+ b\right)\Lambda w
+(1-\gamma_s)w
+(b_s+b^2-\theta)\frac{|y|^2}{4}w-b\left(\frac{\lambda_s}{\lambda}+ b\right)\frac{|y|^2}{2}w
=0,
\end{multline}
where $\theta$ is determined later and the additional term $\theta  |y|^2 w/4 $  is introduced to construct the refined blow-up profile $P(s,y)$ with higher order terms of $\lambda$ and $b$ in Proposition \ref{prop:profile} (In fact, it is related to the solvability of the linearized operators $L_{\pm}$ in \eqref{def:Sjk}), and will modify the modulation equations in \eqref{ode:mod}, and is responsible for the blow-up result obtained in Theorem~\ref{thm:thresh2}. 

\bigskip

Fix $K\in \N$, $K\geq 7 $ is sufficient in the proof of Theorem \ref{thm:thresh2}, and define
$$
 \Sigma_K= \{(j,k)\in \N^2 \ | \ j+k\leq K \}.
$$

\begin{proposition}\label{prop:profile}
Let $\lambda(s)>0$ and $b(s)\in \R$ be $\mathcal C^1$ functions of $s$ such that $\lambda(s)+|b(s)|\ll 1$.
\\
{\rm (i) Existence of a refined blow-up profile.}
For any $(j,k)\in \Sigma_K$, there exist real-valued functions $P_{j,k}^{\pm}\in \mathcal Y$ and $\beta_{j,k}\in\R$ such that
$P(s,y)=\tilde P_K(y;b(s),\lambda(s))$, where $\tilde P_K$ is defined by
\begin{equation}\label{def:P}
   \tilde P_K(y;b,\lambda):=Q(y)+\sum_{(j,k)\in\Sigma_K} b^{2j}\lambda^{k+1} P_{j,k}^+(y)
                +i\sum_{(j,k)\in\Sigma_K}  b^{2j+1}\lambda^{k+1} P_{j,k}^-(y)
\end{equation}
satisfies
\begin{equation}\label{eq:P}
i\partial_s P +\partial^2_y P-P+f(P)+\lambda g(P)
+\theta \frac{|y|^2}{4}P =\Psi_K,
\end{equation}
where $\theta(s)=\tilde \theta(b(s),\lambda(s))$ is defined by
\begin{equation}\label{def:theta}
\tilde \theta(b,\lambda) = \sum_{(j,k)\in\Sigma_K} b^{2j}\lambda^{k+1} \beta_{j,k},
\end{equation}
and $\Psi_K$ satisfies 
\begin{align}
\sup_{y\in \R} \left(e^{\frac{|y|}{2}}\left( |\Psi_K(y)| + |\partial_y \Psi_K(y)|\right)\right) 
 \lesssim \lambda\left(\Big|\frac{\lambda_s}{\lambda}+b \Big|+\left|b_s+b^2-\theta\right|\right) 
 +(|b|^2+\lambda)^{K+2} \label{est:error}.
\end{align}
{\rm (ii) Rescaled blow-up profile.} Let
\begin{equation}\label{def:Pb}
P_b(s,y)=P(s,y)e^{-i\frac{b(s)|y|^2}{4}},
\end{equation}
then we have
\begin{multline}
i\partial_s P_b+\partial^2_y P_b-P_b+f(P_b)+\lambda g(P_b)
-i\frac{\lambda_s}{\lambda}\Lambda P_b\\=-i\left(\frac{\lambda_s}{\lambda}+b\right)\Lambda P_b+(b_s+b^2-\theta)\frac{|y|^2}{4}P_b+\Psi_K e^{-i\frac{b|y|^2}{4}}.
\label{eq:Pb}\end{multline}
{\rm (iii) Mass and energy properties of the  blow-up profile.}
Let
\begin{equation}\label{def:Pblamb}
P_{b,\lambda,\gamma}(s,y) = \frac 1{\lambda^{\frac 12}} P_b\left(s, y\right)e^{i\gamma},
\end{equation}
then we have
\begin{equation}\label{est:Pblamb dmass}
  \left|\frac d{ds}\int_{\R} |P_{b,\lambda,\gamma}|^2\,dy \right|  \lesssim
  \lambda\left(\Big|\frac{\lambda_s}{\lambda}+ b\Big|+\left|b_s+b^2-\theta\right|\right) 
  +(|b|^2+\lambda)^{K+2},
\end{equation}
\begin{equation}\label{est:Pblamb denergy}
  \left|\frac d{ds} {E(P_{b,\lambda,\gamma})}\right|\lesssim
  \frac 1{\lambda^2} \left(\Big|\frac{\lambda_s}{\lambda}+b\Big|+\left|b_s+b^2-\theta\right|
  +(|b|^2+\lambda)^{K+2}\right).
\end{equation}
Moreover, for any $(j,k)\in\Sigma_K,$ there exist  $\eta_{j,k} \in \R$ such that
\begin{equation}\label{est:Pblamb energy}
  \left|E(P_{b,\lambda,\gamma}) - \frac{1}{8}\cdot  \mathcal{E}(b,\lambda)  \cdot \int |y|^2Q^2 dy\right|
  \lesssim \frac{(b^2+\lambda)^{K+2}}{\lambda^2},
\end{equation}
where
\begin{equation}\label{def:Eblamb}
  \mathcal{E}(b,\lambda) =  \frac{b^2}{\lambda^2}
  - 2\beta \frac{\lambda}{\lambda^2} 
  + \frac{\lambda}{\lambda^2} \sum_{(j,k)\in \Sigma_K, j+k\geq 1} b^{2j} \lambda^{k} \eta_{j,k}, \quad \beta=2\mu \frac{ Q(0)^2}{\norm{yQ}_2^2}=4\frac{G(Q)}{\norm{yQ}^2_2}>0.
\end{equation}
\end{proposition}

\begin{remark}\label{rem:blp rate}
Compared with \eqref{eq:w2} and \eqref{eq:P}, the refined blow-up profile $P(s,y)$ with higher order terms in $\lambda$ and $b$ is an approximate solution of $w$ in \eqref{eq:w}, the corresponding blow-up law  is expected from \eqref{eq:w2} and \eqref{est:error}  that
\begin{equation*}
\frac{\lambda_s}{\lambda} + b \approx 0, \quad b_s+b^2-\theta \approx 0,
\end{equation*} 
which shows the effect of the Dirac delta potential perturbation and  has an  approximate solution  in the rescaled variable that $$\lambda_{\rm app} (s)\approx b^2_{\rm app}(s) \approx s^{-2}.$$
(see Lemma \ref{lem:sol:ode app}.)  The above blow-up law is different  with the unperturbed case (i.e. $\mu=0$)
 \begin{equation*}
\frac{\lambda_s}{\lambda} + b =0, \quad   b_s+b^2=0.\ \  \Longrightarrow\ \  \lambda(s)=b(s)=s^{-1}.
\end{equation*}
\end{remark}

The proof follows from the argument in \cite {MeRS14:NLS:blp, RaS11:NLS:mini sol},  we can also refer to \cite{KrLR13:HalfW:nondis, LeMR:CNLS:blp, MaP17:BO:mini sol} and references therein.  

\begin{proof} [Proof of Proposition \ref{prop:profile}]We divide the proof into several steps.
	
\noindent{\bf Step 1: } We firstly consider the construction of the refined blow-up profile $P(s,y)$ in $(i)$. Let
\[
P=Q+\lambda Z, \qquad \quad 
Z=\sum_{(j,k)\in\Sigma_K} b^{2j}\lambda^{k} P_{j,k}^+
+i\sum_{(j,k)\in\Sigma_K}  b^{2j+1}\lambda^{k} P_{j,k}^-,
\]
$$\theta=  \sum_{(j,k)\in\Sigma_K} b^{2j}\lambda^{k+1} \beta_{j,k} ,$$
where $\lambda=\lambda(s)>0$, $b=b(s)$ are time dependent functions,  and $P_{j,k}^{\pm}\in \mathcal Y$ and $\beta_{j,k}\in \R$ are to be determined later such that $P(s,y)$ is an approximate solution of $w$ in \eqref{eq:w} or \eqref{eq:w2} with error estimate \eqref{est:error}. 
Now we set
$$
\Psi_K=i\partial_s P +\partial^2_yP-P+f(P)+ \lambda g(P)
+\theta  \frac{|y|^2}{4}P, \quad 
$$
where $f(P)=|P|^{4}P, \, \ g(P)= \mu \delta P$.

We divide the computation into four steps as follows.

\noindent{\bf Estimate of $i \partial_s P$}. By the definition of $P$, we have
\begin{align*}
i \partial_s P& = i \frac{\lambda_s}{\lambda} \sum_{(j,k)\in\Sigma_K} (k+1) b^{2j} \lambda^{k+1}P_{j,k}^+
+ i b_s \sum_{(j,k)\in\Sigma_K} 2j b^{2j-1}\lambda^{k+1}P_{j,k}^+\nonumber\\
&\quad  -  \frac{\lambda_s}{\lambda} \sum_{(j,k)\in\Sigma_K} (k+1) b^{2j+1}\lambda^{k+1}P_{j,k}^-
- b_s \sum_{(j,k)\in\Sigma_K} (2j+1) b^{2j}\lambda^{k+1}P_{j,k}^-\nonumber\\
&=-i b \sum_{(j,k)\in\Sigma_K} (k+1) b^{2j}\lambda^{k+1}P_{j,k}^{+}- i \left( b^2-\theta \right)   \sum_{(j,k)\in\Sigma_K} 2j b^{2j-1}\lambda^{k+1}P_{j,k}^{+}\nonumber\\
&\quad  + b  \sum_{(j,k)\in\Sigma_K} (k+1) b^{2j+1}\lambda^{k+1}P_{j,k}^{-}
+\left( b^2-\theta \right)  \sum_{(j,k)\in\Sigma_K} (2j+1) b^{2j}\lambda^{k+1}P_{j,k}^{-}+\Psi^{P_s}
\end{align*}
where
\begin{align}
\Psi^{P_s}
&=\left( \frac{\lambda_s}{\lambda}+b\right) \sum_{(j,k)\in\Sigma_K} (k+1) b^{2j}\lambda^{k+1}\left(iP_{j,k}^{+} - b P_{j,k}^-\right)\nonumber \\
&\quad + \left(b_s+b^2-\theta \right) \sum_{(j,k)\in\Sigma_K}  b^{2j-1}\lambda^{k+1} \left(2j i P_{j,k}^+-(2j+1)b P_{j,k}^-\right).
\label{def:Ps}
\end{align}
By the definition of $\theta$ in \eqref{def:theta}, we rewrite  
\begin{align}
i \partial_s P =  -i\sum_{(j,k)\in\Sigma_K} \left((k+1)+2j\right) b^{2j+1}\lambda^{k+1}P_{j,k}^{+}  +& i \sum_{j,k\geq 0} b^{2j+1}\lambda^{k+1}F_{j,k}^{P_s,-} \nonumber \\
& + \sum_{j,k\geq 0} b^{2j}\lambda^{k+1}F_{j,k}^{P_s,+}
+\Psi^{P_s},\label{cal:ps}
\end{align}
where $F_{j,k}^{P_s,\pm}$  are defined for $j,k\geq 0$ by
\begin{align*}
F_{j,k}^{P_s,-} = & \sum^{j+1}_{j'=0} \sum^{k-1}_{k'=0} 2(j-j'+1) \beta_{j',k'}P^{+}_{j-j'+1, k-k'-1}, \\
F_{j,k}^{P_s,+}= &  \sum^{j}_{j'=0} \sum^{k-1}_{k'=0} \Big(2(j-j')+1\Big) \beta_{j',k'}P^{-}_{j-j', k-k'-1} + (k+1)P^{-}_{j-1,k}+(2j-1)P^{-}_{j-1,k},
\end{align*}
which implies that $F_{j,k}^{P_s,\pm}$ only depends on functions $P_{j',k'}^\pm$ and parameters $\beta_{j',k'}$ for 
$(j',k')\in \Sigma_K$ such that either $k'\leq k-1$ and $j'\leq j+1$ or  $k'\leq k$ and $j'\leq j-1$.

\noindent {\bf Estimate of $\partial^2_y P - P + f(P) $}. By the fact that $\partial^2_y Q - Q + f(Q)=0$, we have
\begin{align*}
\partial^2_y P - P + f(P)  =\ &  \partial^2_y\left(Q(y)+\sum_{(j,k)\in\Sigma_K} b^{2j}\lambda^{k+1} P_{j,k}^+(y)
+i\sum_{(j,k)\in\Sigma_K}  b^{2j+1}\lambda^{k+1} P_{j,k}^-(y)\right)\nonumber \\
 & - \left(Q(y)+\sum_{(j,k)\in\Sigma_K} b^{2j}\lambda^{k+1} P_{j,k}^+(y)
 +i\sum_{(j,k)\in\Sigma_K}  b^{2j+1}\lambda^{k+1} P_{j,k}^-(y) \right) \\
 &\quad  +f(Q) +  f(Q+\lambda Z)  - f(Q)  \nonumber
 \\
= & 
- \sum_{(j,k)\in\Sigma_K} b^{2j}  \lambda^{k+1} L_+ P_{j,k}^{+}
-i\sum_{(j,k)\in\Sigma_K} b^{2j+1}\lambda^{k+1} L_- P_{j,k}^{-} \nonumber \\ 
& +f(Q+\lambda Z)-f(Q)-df(Q)\cdot \lambda Z.
\end{align*}
By the structures of $f$ and $P$, we have
\begin{equation*}
f(Q+\lambda Z)-f(Q)-df(Q)\cdot \lambda Z =   \sum_{j\geq 0,k\geq 1} b^{2j}\lambda^{k+1}F_{j,k}^{f,+}+  i \sum_{j\geq 0,k\geq 1} b^{2j+1}\lambda^{k+1}F_{j,k}^{f,-},
\end{equation*}
where  for $j,k\geq 0$, 
$F_{j,k}^{f,\pm}$ depends on $Q$ and on functions $P_{j',k'}^\pm$ for $(j',k')\in \Sigma_K$ such that $k'\leq k-1$ and $j'\leq j$.
Then we obtain that
\begin{multline}\label{cal:f}
\partial^2_y P - P + f(P)   =\  - \sum_{(j,k)\in\Sigma_K} b^{2j}  \lambda^{k+1} L_+ P_{j,k}^++ \sum_{j\geq 0,k\geq 1} b^{2j}\lambda^{k+1}F_{j,k}^{f,+} \\
   -i\sum_{(j,k)\in\Sigma_K} b^{2j+1}\lambda^{k+1} L_- P_{j,k}^{-}  +  i \sum_{j\geq 0,k\geq 1} b^{2j+1}\lambda^{k+1}F_{j,k}^{f,-}  
\end{multline}

\noindent {\bf Estimate of  $\lambda g(P)$}. By the definition of $g$ and $P$, we have
\begin{equation} \label{cal:g}
\lambda g(P)  = \lambda \mu \cdot \delta Q 
+ \sum_{j\geq 0,k\geq 1} b^{2j}\lambda^{k+1}F_{j,k}^{g,+}+
 i \sum_{j\geq 0,k\geq 1} b^{2j+1}\lambda^{k+1}F_{j,k}^{g,-},
\end{equation}
where  $F^{g,\pm}_{j,k}$ are defined  for $j\geq 0, k\geq 1$ by
\begin{align*}
F_{j,k}^{g, +} =  \mu \cdot  \delta P^{+}_{j, k-1},  \quad \text{and}\quad
F_{j,k}^{g,-}= \mu \cdot \delta P^{-}_{j, k-1}. 
\end{align*}

\noindent {\bf Estimate of $\theta \frac {|y|^2}4 P$}. By the definition of $\theta$ and $P$, we obtain
\begin{equation}\label{cal:theta}
\theta \frac {|y|^2}4 P  = \frac {1}4 |y|^2 Q \cdot  \sum_{(j,k)\in\Sigma_K} b^{2j}  \lambda^{k+1} \beta_{j,k}   + \sum_{j\geq 0,k\geq 1} b^{2j}\lambda^{k+1}F_{j,k}^{\theta,+}+ i \sum_{j\geq 0,k\geq 1} b^{2j+1}\lambda^{k+1}F_{j,k}^{\theta,-},
\end{equation}
where 
$F_{j,k}^{\theta,\pm}$ are defined for $j\geq 0, k\geq 1$ by 
\begin{align*}
F_{j,k}^{\theta, +} = &   \sum^{j}_{j'=0}\sum^{k-1}_{k'=0} \beta_{j',k'} \frac{|y|^2}{4} P^{+}_{j-j', k-k'-1},  \\
F_{j,k}^{\theta,-}= &  \sum^{j}_{j'=0}\sum^{k-1}_{k'=0} \beta_{j',k'} \frac{|y|^2}{4} P^{-}_{j-j', k-k'-1},
\end{align*}
which implies that $F_{j,k}^{\theta,\pm}$ only 
depends on functions $P_{j',k'}^\pm$ and parameters $\beta_{j',k'}$ for $(j',k')\in \Sigma_K$ such that $k'\leq k-1$ and $j'\leq j$.

Now by \eqref{cal:ps}, \eqref{cal:f}, \eqref{cal:g} and \eqref{cal:theta}, we obtain
\begin{align*}
\Psi_K  = & 
- \sum_{(j,k)\in\Sigma_K} b^{2j}  \lambda^{k+1} \left( L_+ P_{j,k}^+ - F_{j,k}^{+} -\frac14 \beta_{j,k} |y|^2 Q \right) \\
& - i\sum_{(j,k)\in\Sigma_K} b^{2j+1}\lambda^{k+1} \left( L_- P_{j,k}^- - F_{j,k}^{-} +((k+1) +2j) P_{j,k}^+ \right) \\
& +\Psi^{>K} +\Psi^{P_s},
\end{align*}
where
$
F_{j,k}^{\pm}=F_{j,k}^{P_s,\pm}+F_{j,k}^{f,\pm}+F_{j,k}^{g,\pm}+F_{j,k}^{\theta,\pm},
$
and
$$
\Psi^{>K}= \sum_{j,k\geq 0, \ (j,k)\not\in\Sigma_K} b^{2j}  \lambda^{k+1}  F_{j,k}^{+} 
+ i\sum_{j,k\geq 0, \ (j,k)\not\in\Sigma_K} b^{2j+1}\lambda^{k+1}  F_{j,k}^{-}.
$$
(Note that the series in the expression of $\Psi^{>K}$ contains only a finite number of terms.)
Now, for any $(j,k)\in \Sigma_K$, we want to error term $\Psi_K$ to be sufficiently small, hence  choose recursively $P_{j,k}^\pm \in \mathcal Y$  and $\beta_{j,k}$ to solve the system
\begin{equation}\label{def:Sjk} (S_{j,k})\qquad \left\{
\begin{array}{l}
 L_+ P_{j,k}^+ - F_{j,k}^{+} -\frac14 \beta_{j,k} |y|^2 Q =0, \\
  L_- P_{j,k}^- - F_{j,k}^{-} +\big((k+1) +2j\big) P_{j,k}^+ =0,
\end{array}\right.
\end{equation}
where $F_{j,k}^\pm$ are source terms depending of previously determined $P_{j',k'}^\pm$
and $\beta_{j',k'}$.
We can solve \eqref{def:Sjk} by an induction argument on $j$ and $k$.

For $(j,k)=(0,0)$,  the system $(S_{j,k})$ becomes
\begin{equation}\label{def:S00} (S_{0,0})\qquad \left\{
\begin{array}{l}
 L_+ P_{0,0}^+ - \mu \delta Q - \frac{1}{4} \beta_{0,0} |y|^2 Q =0,
\\
  L_- P_{0,0}^-   +  P_{0,0}^+=0 ,
\end{array}\right.
\end{equation}
By \eqref{eq:Lp}, for any $\beta_{0,0}\in\R$, there exists a unique $P_{0,0}^+\in \mathcal Y$ such that $$L_+ P_{0,0}^+ -\mu \delta Q  -\frac{1}{4} \beta_{0,0} |y|^2 Q =0.$$
In order to solve $P^{-}_{0,0}$ in \eqref{def:S00}, we choose $\beta_{0,0}\in\R$  such that  
\begin{equation*}
0= \psld{P_{0,0}^+}{Q}=-\frac 12\psld{L_+ P_{0,0}^+}{\Lambda Q} =-\frac12 \psld{\mu \delta Q+  \frac{1}{4}\beta_{0,0}|y|^2Q}{ \Lambda Q}
\end{equation*}
where in the second equality we use the fact that $L _+\Lambda Q = -2 Q$ in \eqref{fact:algebra}, which gives
\begin{equation}\label{def:beta}
\beta:=\beta_{0,0}=-4\mu\frac{ \psld{\delta Q}{\Lambda Q}}{\psld{|y|^2Q}{\Lambda Q}}= 2\mu \frac{ Q(0)^2}{\norm{yQ}_2^2}=4 \frac{G(Q)}{\norm{yQ}^2_2}>0.
\end{equation}
Therefore, by \eqref{eq:Ln}, there exists $P_{0,0}^-\in \mathcal Y$ (unique up to the kernel functions $cQ$ of the operator $L_{-}$) such that $$L_- P_{0,0}^-   +  P_{0,0}^+=0.$$

Now, we assume that for some $(j_0,k_0)\in\Sigma_K$, the following conclusion is true:

\noindent {$H(j_0,k_0)$} : {\sl for all $(j,k)\in\Sigma_K$ such that either
$k<k_0$, or $k=k_0$ and $j<j_0$, the system $(S_{j,k})$ has a solution $(P_{j,k}^+,P_{j,k}^-,\beta_{j,k})\in  \mathcal Y \times \mathcal Y \times \R$.} 

By the definition of $F_{j_0,k_0}^\pm$, $H(j_0,k_0)$ implies that $F_{j_0,k_0}^\pm\in\mathcal Y$. We now solve the system $(S_{j_0,k_0})$ as $(S_{0,0})$. By \eqref{eq:Lp}, for any $\beta_{j_0,k_0}\in\R$, there exists a unique $P_{j_0,k_0}^+\in \mathcal Y$ such that $$L_+ P_{j_0,k_0}^+ - F_{j_0,k_0}^{+} - \frac{1}{4}\beta_{j_0,k_0} |y|^2 Q =0.$$ In order to solve $P^{-}_{j_0, k_0}$ in \eqref{def:Sjk},  we choose $\beta_{j_0,k_0}\in\R$ such that  
\begin{equation*}
\psld{ - F_{j_0,k_0}^{-} +((k_0+1) +2j_0) P_{j_0,k_0}^+}{Q}=0,
\end{equation*}
which together \eqref{fact:algebra} implies that 
\begin{equation*}
 \beta_{j_0, k_0}= \frac{1}{\norm{yQ}^2_2 } \left( \big(F^{+}_{j_0, k_0}, \Lambda Q\big)_2 + 2\frac{\big(F^{-}_{j_0, k_0}, Q\big)_2}{(k_0+1)+2j_0} \right) \in \R.
\end{equation*}
By \eqref{eq:Ln}, there exists $P_{j_0,k_0}^-\in \mathcal Y$ (unique up to the kernel functions $cQ$ of the operator $L_{-}$) such that $$L_- P_{j_0,k_0}^-   - F_{j_0,k_0}^{-} +((k_0+1) +2j_0) P_{j_0,k_0}^+=0.$$
In particular, we have proved that
if $j_0<K$, then $H(j_0,k_0)$ implies $H(j_0+1,k_0)$, and $H(K,k_0)$ implies $H(0,k_0+1)$.

The induction argument on $(j,k)$ can solve \eqref{def:Sjk} in $ \mathcal Y \times \mathcal Y \times \R$ for $(j,k)\in \Sigma_K$.
Therefore, we can obtain the refined blow-up profile $P(s,y)$ according to the definition in \eqref{def:P}.

It remains to estimate $\Psi_K$ and $\partial_y \Psi_K$.
It is easy to check that
\begin{align*}
& \sup_{y\in \R} \left(e^{\frac{|y|}{2}}\left( |\Psi^{P_s}(y)|+  |\partial_y\Psi^{P_s}(y)|\right)\right) \lesssim  \lambda \left(\left| \frac{\lambda_s}{\lambda}+b\right|+
\left|b_s+b^2-\theta\right|\right),
\end{align*}
and
\begin{equation*}
 \sup_{y\in \R}  \left(e^{\frac{|y|}{2}}\left(|\Psi^{>K}|+|\partial_y \Psi^{>K}| \right) \right)\lesssim \left( |b|^{2(K+1)}\lambda+ \lambda^{K+2} \right)\lesssim \left( |b|^{2}+ \lambda\right)^{K+2}.
\end{equation*}
This completes the proof of $(i)$.

\noindent{\bf Step 2: } By the definition of $P_b$ in \eqref{def:Pb}, we have
\begin{equation*}
\partial_s P =\left(\partial_s P_b + i b_s\frac{|y|^2}{4} P_b \right)e^{ib\frac{|y|^2}{4}}, 
\end{equation*}
\begin{equation*}
\partial^2_y P = \left( \partial^2_y P_b + ib\Lambda P_b -b^2\frac{|y|^2}{4} P_b
 \right)e^{ib\frac{|y|^2}{4}}.
\end{equation*}
By the property of the Dirac delta operator, we have
$$\mu \cdot \delta P(s)=g(P) =g(P_b) \cdot e^{i\frac{|y|^2}{4}}=\mu\cdot  \delta P_b(s) \cdot e^{i\frac{|y|^2}{4}}.$$
Inserting these equalities into \eqref{eq:P}, we can obtain \eqref{eq:Pb}.

\noindent{\bf Step 3: } To prove \eqref{est:Pblamb dmass}, we use  \eqref{def:Pblamb} and multiply \eqref{eq:Pb} with $iP_b$ to get
\begin{equation*}
\frac12 \frac{d}{ds}\|P_{b,\lambda, \gamma}\|_{2}^2 = \frac12 \frac{d}{ds}\|P_b\|_{2}^2=(i\partial_sP_b,iP_b)_2=(\Psi_K e^{-i\frac{b|y|^2}{4}},iP_b)_2,
\end{equation*}
where  in the last equality we use the identity $(P_b,\Lambda P_b)_2=0$. Therefore, \eqref{est:Pblamb dmass} follows from \eqref{est:error}. 

To prove \eqref{est:Pblamb denergy}, we obtain from scaling and \eqref{delta:scal} that
\begin{equation*}
 E(P_{b,\lambda,\gamma})=\frac{1}{\lambda^2}\left(\frac 12\int_{\R}|\partial_y P_b|^2\, dy-\int_{\R} F(P_b)\, dy-\lambda \int_{\R} G(P_b)\, dy\right)=:\frac{1}{\lambda^2}\tilde E(\lambda, P_b)
\end{equation*}
Therefore, we have
\begin{equation}\label{eq:Pblamb:cal1}
\frac{d}{ds} E(P_{b,\lambda,\gamma})=\frac{1}{\lambda^2}\left(\dual{\tilde E'(\lambda,P_b)}{\partial_sP_b}-2\frac{\lambda_s}{\lambda}\tilde E(\lambda,P_b)- \lambda \frac{\lambda_s}{\lambda} \int_{\R} G(P_b)\, dy \right).
\end{equation}
By \eqref{eq:Pb} and the fact that
\begin{equation*}
\dual{\tilde E'(\lambda, P_b)}{i\left(\partial^2_y P_b - P_b + f(P_b) + \lambda g(P_b)\right)} = 0,
\end{equation*}
 we have
\begin{multline}\label{eq:Pblamb:cal2}
\dual{\tilde E'(\lambda,P_b)}{\partial_sP_b}
=\frac{\lambda_s}{\lambda}\dual{\tilde E'(\lambda,P_b)}{\Lambda P_b}-\left(\frac{\lambda_s}{\lambda}+b\right)\dual{\tilde E'(\lambda,P_b)}{\Lambda P_b}
\\+(b_s+b^2-\theta)\dual{i\tilde E'(\lambda,P_b)}{\frac{|y|^2}{4}P_b}
+\dual{i\tilde E'(\lambda,P_b)}{\Psi_Ke^{-i\frac{b|y|^2}{4}}}.
\end{multline}
By integration by parts, we obtain
\begin{align}
\dual{\tilde E'(\lambda,P_b)}{\Lambda P_b}=&\, \int_{\R} |\partial_y P_b|^2\,dy-2\int_{\R} F(P_b)\,dy-\lambda \int_{\R} G(P_b)\,dy \nonumber \\
=&\, 2\tilde E(\lambda,P_b)+\lambda\int_{\R} G(P_b)\,dy,\label{eq:Pblamb:cal3}
\end{align}
by \eqref{eq:Pblamb:cal1}, \eqref{eq:Pblamb:cal2}\eqref{eq:Pblamb:cal3}, and \eqref{est:error}, we have
\begin{align*}
\frac{d}{ds} E(P_{b,\lambda,\gamma})= 
&\ \frac{1}{\lambda^2}\left[\frac{\lambda_s}{\lambda} \left(2\tilde E(\lambda,P_b)+\lambda\int_{\R} G(P_b)\,dy\right)-2\frac{\lambda_s}{\lambda}\tilde E(\lambda,P_b)-\lambda \frac{\lambda_s}{\lambda} \int_{\R} G(P_b)\,dy \right]\\
&\ +  \frac{1}{\lambda^2}O\left(\left| \frac{\lambda_s}{\lambda} +b\right|+|b_s+b^2-\theta|+(b^2+\lambda)^{K+2}\right)\\
= & \ \frac{1}{\lambda^2}O\left(\left| \frac{\lambda_s}{\lambda} +b\right|+|b_s+b^2-\theta|+(b^2+\lambda)^{K+2}\right),
\end{align*}
which implies \eqref{est:Pblamb denergy}. 

To show \eqref{est:Pblamb energy}. By \eqref{def:Pb}, we compute $E(P_{b,\lambda, \gamma})$  as follows
\begin{align*}
\lambda^2 E(P_{b,\lambda,\gamma})&=\frac 12\int_{\R}|\partial_y P_b|^2-\int_{\R} F(P_b)-\lambda \int_{\R} G(P_b)\\
&=\frac 12\int_{\R} |\partial_y P|^2\,dy-\int_{\R} F(P)\,dy+\frac{1}{8}b^2\int_{\R}|y|^2|P|^2\,dy-\lambda \int_{\R} G(P)\,dy.
\end{align*}
Thus, by replacing $P=Q+\lambda Z$, we have
\begin{align}
 \lambda^2 E(P_{b,\lambda,\gamma})
=\ & \frac 12\int_{\R} |\partial_y Q|^2\,dy-\int_{\R} F(Q)\,dy+\frac{1}{8}b^2\int_{\R}|y|^2Q^2\,dy-\lambda \int_{\R} G(Q)\,dy \nonumber \\
& +\lambda \int_{\R} \big(-\partial^2_y Q-f(Q)\big)\cdot \Re Z\,dy-\lambda^{2}\int_{\R} g(Q) \cdot \Re Z\,dy+\frac{1}{4}b^2 \lambda\int_{\R}|y|^2Q \cdot \Re Z\,dy \nonumber \\
&+\frac{1}{2}\lambda^{2} \int_{\R}|\partial_y Z|^2\,dy+\frac{1}{8}b^2\lambda^{2}\int_{\R}|y|^2|Z|^2\,dy\nonumber \\
& 
-\int_{\R}\Big(F(Q+\lambda Z)-F(Q)-\lambda f(Q) \cdot \Re Z\Big)\,dy\nonumber \\
&-\lambda\int_{\R}\Big(G(Q+\lambda Z)-G(Q)-\lambda g(Q) \cdot \Re Z\Big)\,dy. \label{eq:Pblamb:cal4}
\end{align}

On the one hand, by the Pohozaev identity, we have
\begin{equation}\label{est:Z0}
\frac 12\int_{\R} |\partial_y Q|^2\,dy-\int_{\R} F(Q)\,dy=0,
\end{equation}
and by the definition \eqref{def:beta} of $\beta_{0,0}$, we have 
\begin{equation}\label{est:Z1}
\lambda \int_{\R} G(Q)\,dy=\lambda \frac{\beta}{4}\int_{\R} |y|^2Q^2\,dy=\frac {1}{8} \cdot 2\beta\lambda \int_{\R}|y|^2Q^2\,dy.
\end{equation}

On the other hand,  by the facts that  
\begin{equation*}
\partial^2_y Q+f(Q)=Q,\quad\text{and}\quad  \int_{\R} Q \cdot P_{0,0}^+ \,dy=0,
\end{equation*}
we have
\begin{equation}\label{est:Z2f}
\lambda \int_{\R}\big(-\partial^2_y Q-f(Q)\big)\cdot \Re Z\,dy =  -
\lambda \int_{\R} Q \cdot \Re Z\,dy= \lambda \sum_{(j,k)\in \Sigma_K,j+k\geq 1} b^{2j}\lambda^{k}\eta_{j,k}^{\rm I},
\end{equation}
for some $\eta_{j,k}^{\rm I}\in \R$, and
\begin{equation}\label{est:Z2g}
\lambda^{2}\int_{\R} g(Q) \cdot  \Re Z\,dy=\lambda \sum_{(j,k)\in \Sigma_K,k\geq 1}b^{2j}\lambda^{k}\eta_{j,k}^{\rm II},
\end{equation}
for some $\eta_{j,k}^{\rm II}\in \R$, and
\begin{equation}\label{est:Z2conf} 
b^2 \lambda\int_{\R} |y|^2Q \cdot \Re Z\,dy=\lambda \sum_{(j,k)\in \Sigma_K, j\geq 1}b^{2j}\lambda^{k}\eta_{j,k}^{\rm III},
\end{equation}
for some $\eta_{j,k}^{\rm III}\in \R$, and
\begin{equation}\label{est:Zhconf}
\frac{1}{2}\lambda^{2}\int_{\R} |\partial_y Z|^2\,dy+\frac{1}{8}b^2\lambda^{2}\int_{\R} |y|^2Z^2\,dy
=\lambda \sum_{(j,k)\in \Sigma_K,j+k\geq 1}b^{2j}\lambda^{k}\eta_{j,k}^{\rm IV},
\end{equation}
for some $\eta_{j,k}^{\rm IV}\in \R$.
Moreover, by Taylor expansion as before, for some $\eta_{j,k}^{\rm V},\eta_{j,k}^{\rm VI}\in \R$
\begin{align}\label{est:ZhF}
 \left|\int_{\R}\Big(F(Q+\lambda Z)-F(Q)-\lambda f(Q)\cdot \Re Z\Big) \,dy 
-\lambda\sum_{(j,k)\in \Sigma_K, k\geq 1}b^{2j}\lambda^{k}\eta_{j,k}^{\rm V} \right| 
 \lesssim \left(b^2+\lambda\right)^{K+2},
 \end{align}
 and
\begin{align}\label{est:ZhG}
  \left|\lambda\int_{\R}\Big(G(Q+\lambda Z)-G(Q)-\lambda g(Q)\cdot \Re Z\Big) \,dy 
-\lambda\sum_{(j,k)\in \Sigma_K,k\geq 1}b^{2j}\lambda^{k}\eta_{j,k}^{\rm VI} \right| 
  \lesssim \left(b^2+\lambda\right)^{K+2}.
 \end{align}
  
By inserting \eqref{est:Z0} - \eqref{est:ZhG} into \eqref{eq:Pblamb:cal4}, we can prove \eqref{est:Pblamb energy} and complete the proofs of Proposition \ref{prop:profile}.
\end{proof}

\subsection{Approximate blow-up law}\label{subsect:blp}
As shown in Proposition \ref{prop:profile}, $P(s,y)$ can be viewed as the refined blow-up profile of \eqref{eq:dnls} when  $b(s)$ and $\lambda(s)$ satisfy the smallness condition $|b|+\lambda \ll 1$ and 
 the  blow-up law 
\begin{equation*}
\frac{\lambda_s}{\lambda} + b \approx 0, \quad b_s+b^2-\theta \approx 0
\end{equation*} 
where $\displaystyle  \theta= \sum_{(j,k)\in\Sigma_K} b^{2j}\lambda^{k+1} \beta_{j,k}.$  We now look for a  solution to  the following approximate  system
\begin{equation}\label{eq:ode appsys}
\quad
\frac{\lambda_s}{\lambda} + b=0, \quad b_s+b^2-\beta\lambda=0,  
\end{equation}
where $\beta=\beta_{0,0}= 2\mu \frac{ Q(0)^2}{\norm{yQ}_2^2}>0$ in this subsection. Indeed, for $|b|+\lambda \ll 1$,  the first term $\beta\lambda$  in $\theta$ is the main term in $\theta$, and the only term in $\theta$ that will modify the blow-up speed.

\begin{lemma}\label{lem:sol:ode app}
Let
\begin{equation}\label{sol:ode app}
\lambda_{\rm app}(s)=2/(\beta s^2), 
\quad b_{\rm app}(s)= 2/s,
\end{equation}
then $(\lambda_{\rm app}(s),b_{\rm app}(s))$ solves \eqref{eq:ode appsys} for $s>0$.
\end{lemma}

\begin{proof}
By simple computation, we have
$$
\left(\frac{b^2}{\lambda^2}\right)_s=2\frac{b}{\lambda^2}\left(b_s -\frac{\lambda_s}{\lambda}b \right)=2\frac{b}{\lambda}\frac{b_s+b^2}{\lambda}
=2\beta\frac{b}{\lambda}=-\frac{2\beta}{\lambda} \frac{\lambda_s}{\lambda} =\left(\frac{2\beta}{\lambda}\right)_s
,$$ 
and so
\begin{equation}\label{const}
\frac{b^2}{\lambda^2}-\frac{2\beta}{\lambda}=c_0.
\end{equation}
Taking $c_0=0$, and using $b=-\frac{\lambda_s}{\lambda}>0$, we obtain
$$
-\frac{\lambda_s}{\lambda^{1+\frac{1}{2}}}=\sqrt{2\beta}.
$$
This implies that
\begin{equation*}
 \lambda(s)=\frac{2}{\beta}s^{-2},
\quad b(s)=2s^{-1} 
\end{equation*}
is solution of \eqref{eq:ode appsys} and completes the proof.
\end{proof}

\begin{remark}\label{rem:tapp} We now give some remarks on the approximate system \eqref{eq:ode appsys}.
\begin{enumerate}
\item After we obtain the  solution $\big(\lambda_{\rm app}(s),\ b_{\rm app}(s)\big)$ to the  approximate system \eqref{eq:ode appsys} in the time rescaled variable $s$ in Lemma \ref{lem:sol:ode app}, 
we now express this solution in the  time variable $t_{\rm app}$ related to $\lambda_{\rm app}$. Let $t_{app}(s)<0$ and
\begin{equation*}
 {dt_{\rm app}}= {\lambda_{\rm app}^2} ds=\frac{4}{\beta^2} s^{-4}ds,
\end{equation*}
then we have
\begin{equation}
\label{def:tapp}
t_{\rm app}(s) =- C_s s^{-3} \quad \hbox{where} \quad C_s = 4/(3\beta^2), 
\end{equation}
where we use the convention that $t_{\rm app}(s)\to 0^-$ as $s\to +\infty$. As a consequence, we obtain  
\begin{equation}\label{def:lapp}
\lambda_{\rm app}(t_{\rm app})=\lambda_{app}(t_{app}(s))=\frac{2}{\beta}s^{-2}=C_\lambda |t_{\rm app}|^{\frac{2}{3}} \quad \hbox{where} \quad 
C_\lambda  =\frac{2}{\beta} C_s^{-2/3},
\end{equation} 
\begin{equation}\label{def:bapp}
b_{\rm app}(t_{\rm app})=b_{app}(t_{app}(s))=2s^{-1}=C_b|t_{\rm app}|^{\frac{1}{3}}, \quad
\hbox{where} \quad C_b = 2 C_s^{-1/3}.
\end{equation}
\item  For the approximate system \eqref{eq:ode appsys}, we show the $(\lambda, b)$ flows driven by the vector fields $(-\lambda b, -b^2 + \beta \lambda)$ with different $\beta\in \R$ as those in Figure \ref{figure}. In case (a), we denot the curve $\beta \lambda = b^2$ by the blue dot curve. From these pictures, we can obtain the heuristic that there exist finite time blow-up solutions (corresponding to $\lambda(s) \to 0+$ as $s\to \infty$) only for case (a) $\beta>0$ (the attractive delta potential case) and case (c) $\beta=0$ (the $L^2$-critical NLS).

\begin{figure}[htbp]
\centering
\subfigure[$\beta>0$]{
\begin{minipage}[t]{0.45\linewidth}	
\centering
\includegraphics[width=8cm]{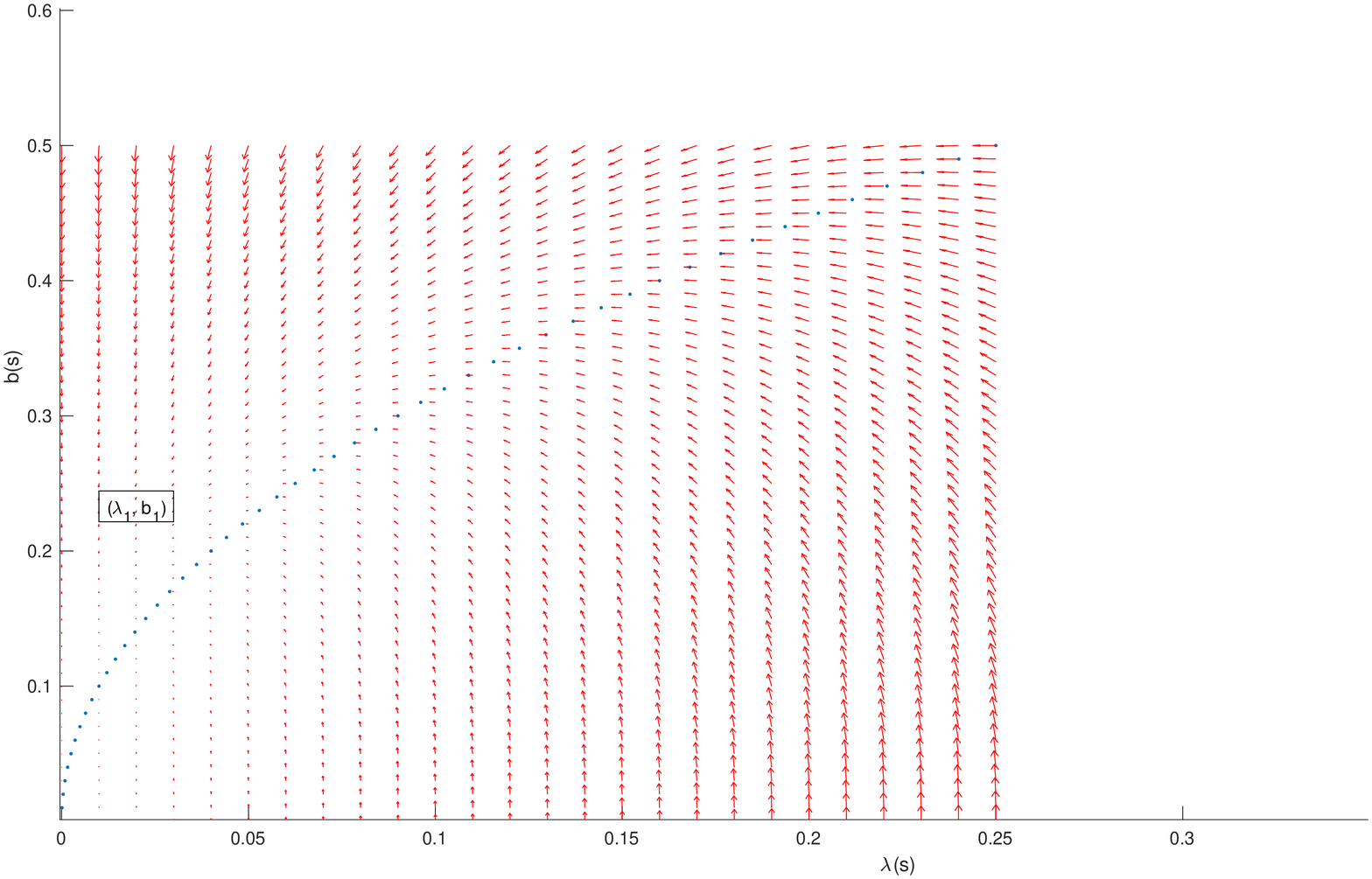}
\end{minipage}
}
\subfigure[$\beta<0$]{
	\begin{minipage}[t]{0.45\linewidth}
		\centering
		\includegraphics[width=8cm]{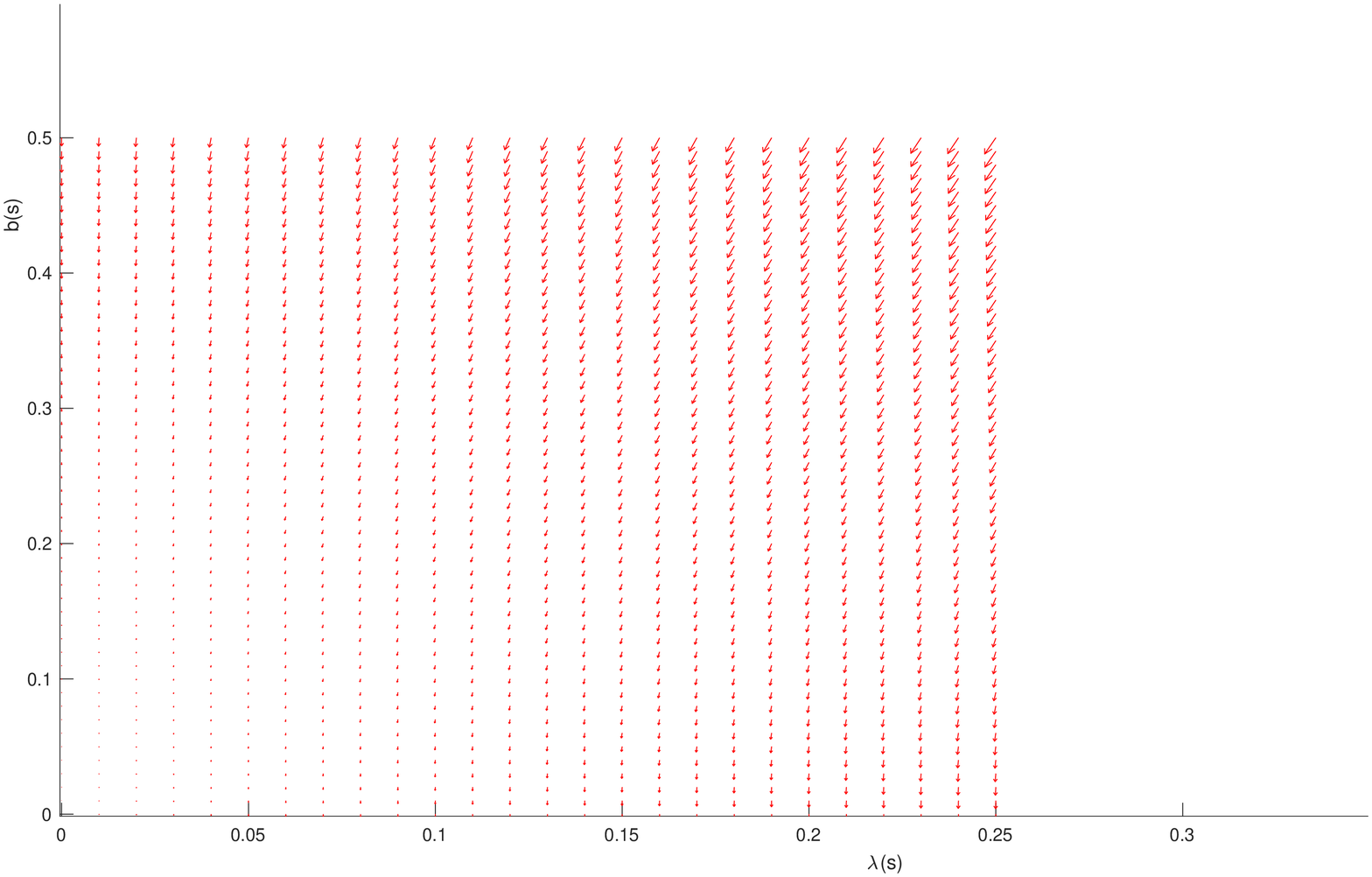}
	\end{minipage}
}

\subfigure[$\beta=0$]{
	\begin{minipage}[t]{0.45\linewidth}
		\centering
		\includegraphics[width=8cm]{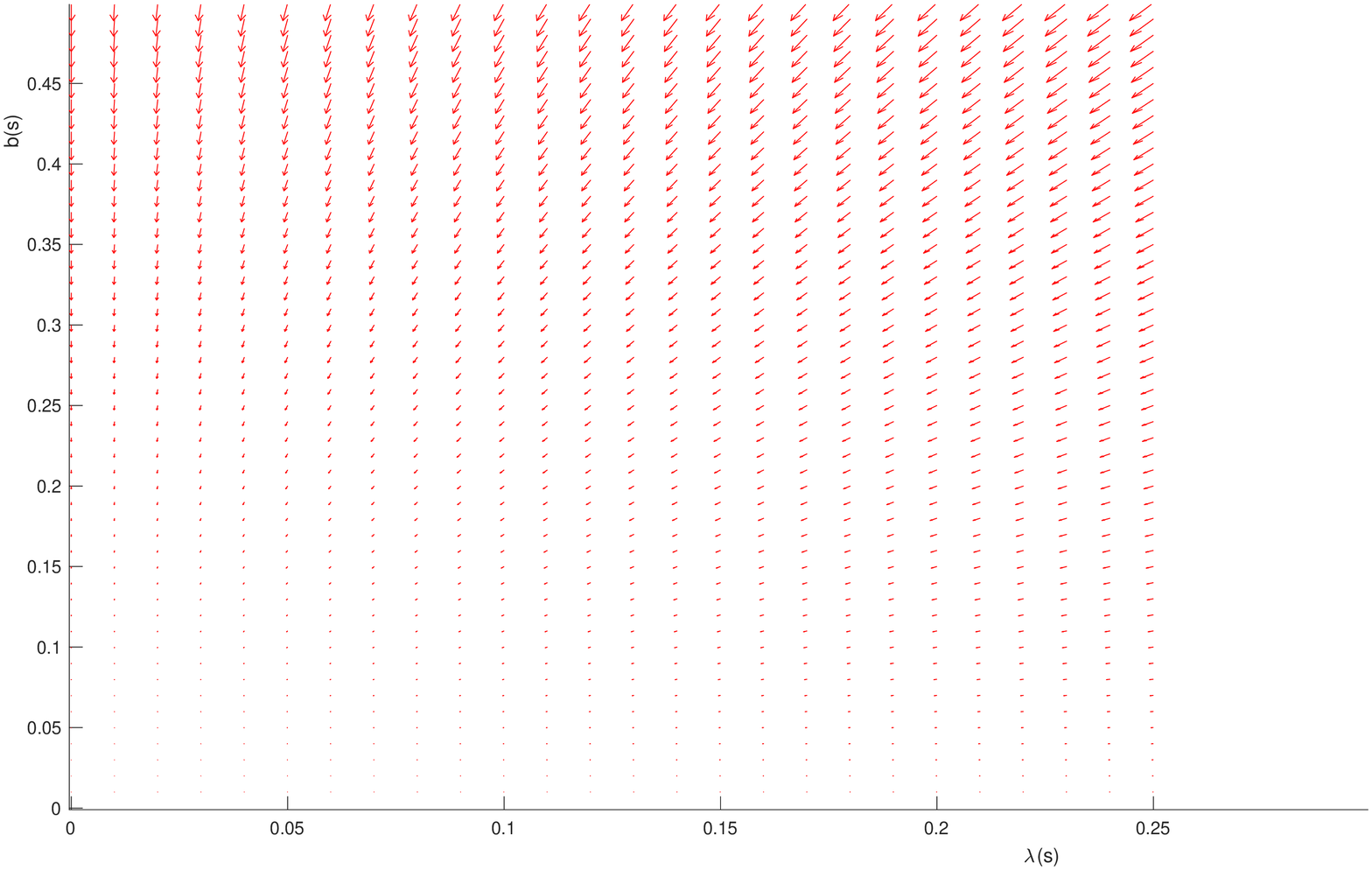}
	\end{minipage}
}
\centering
\caption{$(\lambda, b)$ flows driven by the vector fields $(-\lambda b, -b^2 + \beta \lambda)$ with different  $\beta\in \R$. }
\label{figure}
\end{figure}
\end{enumerate}
\end{remark}

In order to adjust the value of the energy of $P_{b,\lambda,\gamma}$  in \eqref{est:Pblamb energy} up to small error, and to close the bootstrap argument for $(\lambda, b)$  at the end of the proof of Proposition \ref{prop:unif est}, we will choose suitable final conditions $\lambda_1$ and $b_1$ of $(\lambda, b)$ at sufficiently large rescaled time $s_1$ as the base case (see \eqref{eq:final data} and \eqref{est:base case}). Let $E_0\in \R$ and
\begin{equation*}
C_0=\frac{8E_0}{\displaystyle \int_{\R} |y|^2Q^2\, dy}.
\end{equation*}
Fix $0<\lambda_0\ll 1$  such that\footnote{We can take any $E_0\in \R$ because of the fact that $\beta>0$ for $\mu>0$ here.} $2\beta  + C_0 \lambda_0 >0$,  and for $\lambda\in (0,\lambda_0]$,  we define the auxiliary function
\begin{equation}\label{def:F}
\mathcal F(\lambda)=\int_\lambda^{\lambda_0}\frac{d\tau}{\tau^{\frac{1}{2}+1}\sqrt{2\beta+C_0\tau}} \quad \Longrightarrow \quad  \frac{d}{d\lambda} \mathcal F(\lambda) = -  \frac{1}{\lambda^{\frac{1}{2}+1}\sqrt{2\beta+C_0 \lambda}},
\end{equation}
where $\mathcal F$ is related to the resolution of $\lambda$ and $b$ for  the system \eqref{const}
with $c_0=C_0$. (See \eqref{est:b} and \eqref{est:dsF} in  the proof of Proposition \ref{prop:bootstrap} in Subsection \ref{sect:boot pf} for more details.)

\begin{lemma}\label{lem:initial conds}
Let $s_1\gg1$, then there exist sufficiently small  $\lambda_1>0$ and $b_1$ such that
\begin{align}
\label{est:lamb b:ini}
& \left|\frac {\lambda_1^{\frac 1 2}}{\lambda_{\rm app}^{\frac 1 2}(s_1)} - 1 \right|
+ \left|\frac{b_1}{b_{\rm app}(s_1)}-1 \right|\lesssim s_1^{-1},
\\& 
\mathcal{F}(\lambda_1)=s_1,\quad 
\mathcal{E}(b_1,\lambda_1)=C_0.\label{est:energy:ini}
\end{align}
\end{lemma}

\begin{proof}
For sufficiently large $s_1$,  we firstly choose $\lambda_1$. By \eqref{def:F}, $\mathcal F$ is a decreasing function of $\lambda$ satisfying
$\mathcal F(\lambda_0)=0$ and $\displaystyle \lim_{\lambda \to 0} \mathcal F(\lambda)=+\infty$.
Thus there exists a unique $\lambda_1\in (0,\lambda_0)$ such that
$\mathcal F(\lambda_1)=s_1$.

For $\lambda\in (0,\lambda_0]$, we have
\begin{align}\label{est:F}
\left|\mathcal F(\lambda)-\frac{2}{\sqrt{2\beta} \lambda^{\frac{1}{2}}} \right| \lesssim 1+\left|\int_\lambda^{\lambda_0}\frac{d\tau}{\tau^{\frac{1}{2}+1}}\left[\frac{1}{\sqrt{2\beta +C_0\tau}}-\frac{1}{\sqrt{2\beta}}\right]\right|
\lesssim 1.
\end{align}
If taking $\lambda=\lambda_1$, we obtain from $\mathcal F(\lambda_1)=s_1$ and $\lambda_{\rm app}(s)=2/(\beta s^2)$ that
$$
\left|s_1 - \frac{2}{\sqrt{2\beta} \lambda_1^{\frac{1}{2}}} \right|
\lesssim 1 \quad \Longleftrightarrow \quad
\left|\frac {\lambda_1^{\frac 12}}{\lambda_{\rm app}^{\frac 12}(s_1)} - 1 \right|
\lesssim s_1^{-1}.
$$

Secondly, we can choose  $b_1$ sufficiently small.
From the definition of $\mathcal{E}(b, \lambda)$ and $\lambda_{\rm app}(s) = 2/(\beta s^2)$, we define the function $h(b)$ as follows
\begin{align*}
h(b):=\lambda_1^2 \mathcal{E}(b,\lambda_1)
& =b^2-\left(\frac{2}{ s_1}\right)^2 - 2\beta \left(\lambda_1
-\lambda_{\rm app}(s_1)\right)
+\lambda_1 \sum_{(j,k)\in\Sigma_K,\  j+k\geq 1} b^{2j}\lambda_1^{k}  \eta_{j,k}\\
& =\big(1+O(s^{-2}_1)\big)b^2-\left(\frac{2}{ s_1}\right)^2  +O(s_1^{- 3}),
\end{align*}
where $b$ is close to $b_{\rm app}(s_1)$. By $b_{\rm app}(s)=2/s$, we have 
$$
|h(b_{\rm app}(s_1))|\lesssim  s_1^{-3},\quad |h'(b_{\rm app}(s_1))| \geq 2b_{\rm app}(s_1)+O(s_1^{-3})
\geq {s_1^{-1}}.
$$
Since $C_0 \lambda_1^2\approx s_1^{- 4}$, it follows from the implicit function theorem that
 there exists a unique $b_1$ such that
$$
 h(b_1)=C_0 \lambda_1^2, \quad \text{where}\quad 
|b_1-b_{\rm app}(s_1)|\lesssim s^{-3}_1/ s^{-1}_1 \lesssim    s_1^{-2},$$
which implies 
$
\mathcal{E}(b_1,\lambda_1)= h(b_1)/\lambda^2_1= C_0,
$ and completes the proof.
\end{proof}

\section{Uniform estimates in the rescaled time variable}\label{Sect3: uniform ests}
After the construction of the refined blow-up profile in Proposition \ref{prop:profile} and final data setup on modulation parameters $\lambda, b$ at the rescaled time $s=s_1$ in Lemma \ref{lem:initial conds},  we will show the uniform backwards estimates of $(\lambda, b, \varepsilon)$ for specific solutions of \eqref{eq:dnls} on $[s_0, s_1]$ in Proposition \ref{prop:unif est}, where $s_0$ is sufficiently large, but independent of $s_1$. These uniform backwards estimates will play key role to construct minimal mass blow-up solution of \eqref{eq:dnls} by the standard compact argument in next section. 

Let $P$, $P_b$ be defined by \eqref{def:P}, \eqref{def:Pb} in Proposition \ref{prop:profile}, $\rho$ be given by \eqref{fact:algebra} and define 
\begin{equation}\label{def:rhob}
\rho_b(s,y):=\rho(y)e^{-i\frac{b(s)|y|^2}{4}}.
\end{equation}

We firstly recall the following standard modulation decomposition of the solution $u(t)$ in the small tube of $Q$.
\begin{lemma}\label{lem:modul} There exists $\delta_0>0$ small enough such that for any $u(t)\in \mathcal C(I,H^1(\R))$  satisfying
\begin{equation}\label{est:small tube}
\sup_{t\in I}\inf_{\lambda_0>0,\gamma_0}\left\|\lambda_0^{\frac 12} u(t,\lambda_0 y)e^{i\gamma_0}- Q(y) \right\|_{H^1} \leq \delta_0,
\end{equation}
 there exist $\mathcal C^1$ functions $\lambda\in(0,+\infty)$, $b\in\R$, $\gamma\in\R$ on $I$ such that $u$ admits 
 a unique decomposition of the form 
\begin{equation}\label{eq:refined modul}
u(t,x)=\frac{1}{\lambda^{\frac{1}{2}}(t)}\left(P_{b(t)}+\eps(t,y)\right)e^{i\gamma(t)},\qquad 
y=\frac{x}{\lambda(t)},
\end{equation}
where the remainder $\eps$ obeys the following orthogonality structure:
\begin{equation}\label{eq:orth struct}
 \forall \ t\in I, \quad \psld{\eps}{i\Lambda P_b}=\psld{\eps}{|y|^2P_b}=\psld{\eps}{i\rho_b }=0. 
\end{equation} 
\end{lemma}
\begin{proof}
	 Please refer to \cite{MeR05:NLS:blp} for details.
	\end{proof}

Let $E_0\in \R$  and $t_1<0$ be close to $0$. By Remark \ref{rem:tapp}, we set up the initial rescaled time $s_1$ as 
\begin{equation*}
s_1:=\left|C_s^{-1}t_1\right|^{-1/3} \gg 1 \quad \Longleftrightarrow\quad  t_1 = t_{\rm app}(s_1). 
\end{equation*} 
Let $\lambda_1$ and $ b_1$ be given by Lemma \ref{lem:initial conds}. Let $u(t)$ be the solution of $\eqref{eq:dnls}$ for $t\leq t_1$ with final data 
\begin{equation}\label{eq:final data}
u(t_1,x)=\frac{1}{\lambda_1^{\frac{1}{2}}}P_{b_1}\left(\frac{x}{\lambda_1}\right).
\end{equation}
As long as the solution $u(t)$ satisfies  \eqref{est:small tube}, we consider its modulation decomposition $(\lambda,b,\gamma,\varepsilon)$ from Lemma \ref{lem:modul} and define the rescaled time variable $s$ related to $\lambda$ by 
\begin{equation}
\label{def:st}
s= s_1-\int_{t}^{t_1} \frac{1}{\lambda^2(\tau)}d\tau \quad \Longleftrightarrow \quad \frac{ds}{dt}=\frac{1}{\lambda^2(s)} \quad\text{and}\quad  s_1 = s(t_1).
\end{equation}

The main result in this subsection is the following uniform backwards estimates on the decomposition of $u(s)$ on the rescaled time interval $[s_0,s_1]$, where $s_0$ is sufficiently large, but independent of $s_1$.

\begin{proposition}\label{prop:unif est}
Let $\mu>0$ and $K\geq 7$, there exists sufficiently large $s_0>0$, which is independent of $s_1$ such that
the solution $u$ of $\eqref{eq:dnls}$ with final data \eqref{eq:final data}
exists and  satisfies \eqref{est:small tube} on the rescaled time interval $[s_0,s_1]$. Moreover, its modulation decomposition
\begin{equation*}
u(t,x)=\frac{1}{\lambda^{\frac{1}{2}}(s)}
\left(P_b+\eps\right)\left(s,y\right) e^{i\gamma(s)},
\quad \text{where}\quad \frac{ds}{dt} = \frac{1}{\lambda^2(s)},\quad 
y=\frac{x}{\lambda(s)},
\end{equation*}
satisfies the following uniform estimates
\begin{equation}\label{est:small est:s}
\norm{\eps(s)}_{H^1}\lesssim s^{-(K+1)},\quad
\left|\frac{\lambda^{\frac  12}(s)}{\lambda_{\rm app}^{\frac 12}(s)}-1\right|+
\left|\frac{b(s)}{b_{\rm app}(s)}-1\right|  
\lesssim s^{-1}
\end{equation}
on $ [s_0,s_1]$. In addition, we have for any $s\in [s_0, s_1]$ that
\begin{equation}\label{est:Pblamb energy2}
|E(P_{b,\lambda,\gamma}(s))-E_0|\leq O(s^{-K+3}).
\end{equation}
\end{proposition}

In the rest of this part, we will use a bootstrap argument involving the following estimates 
\begin{equation}
\label{est:small est:rough}\norm{\eps(s)}_{H^1} < s^{-K},\quad
\left|\frac{\lambda^{\frac12}(s)}{\lambda_{\rm app}^{\frac12}(s)}-1\right|
+\left|\frac{b(s)}{b_{\rm app}(s)}-1\right|  <s^{-\frac 12}
\end{equation}
to show \eqref{est:small est:s} and complete the proof of Proposition \eqref{prop:unif est}.

\subsection{Bootstrap argument}\label{sect:boot base}

For $s_0>0$  to be chosen large enough (independently of $s_1$),   we 
define
\begin{equation}\label{def:time interval}
s_*=\inf\{ \tau\in[s_0,s_1];\eqref{est:small est:rough} \text{ holds on }  [\tau,s_1]\}.
\end{equation}
By \eqref{est:lamb b:ini} and \eqref{eq:final data},  the base case for the induction argument
\begin{equation}\label{est:base case}
\varepsilon(s_1) \equiv 0,  \quad 
\left|\frac{\lambda_1^{\frac12}}{\lambda_{\rm app}^{\frac12}(s_1)}-1\right|+\left|\frac {b_1}{b_{\rm app}(s_1)}-1\right|\lesssim s_1^{-1} < s^{-\frac12}_1
\end{equation}  
holds for $s_1$ large,  hence $s_*$ is well-defined and $s_*<s_1$ by the continuity of the solution of \eqref{eq:dnls} in $H^1(\R)$. 
After the preparations in Subsection \ref{sect:modul} and Subsection \ref{sect:Mono form},  we will prove that the estimates \eqref{est:small est:s} and \eqref{est:Pblamb energy2} hold on $[s_*,s_1]$ in Subsection \ref{sect:boot pf}. Then we obtain $s_*=s_0$ from the bootstrap argument and complete the proof of Proposition \ref{prop:unif est}. Two key ingredients in the proof are the follows:
\begin{enumerate}
	\item  Dynamical estimate of the parameters $(\lambda, b)$ in Subsection \ref{sect:modul} by deriving the modulation equations from the orthogonal structure of $\varepsilon$ in Lemma \ref{lem:modul}, the mass conservation law and the construction of $P_b$ in Proposition \ref{prop:profile}. 
	\item Dynamical estimate of the remainder $\varepsilon$  in Subsection \ref{sect:Mono form} by combining the  Energy-Morawetz functional with the coercivity property of the linearized operators $L_{\pm}$.  The Energy-Morawetz estimate was firstly used in \cite{RaS11:NLS:mini sol}. 
\end{enumerate}

\subsection{Modulation equations}\label{sect:modul}

In this part, we work with the solution $u(t)$ of Proposition \ref{prop:unif est} on the rescaled time interval $[s_*,s_1]$ (i.e.  $\varepsilon$ and $\lambda, b$ satisfy \eqref{est:small est:rough} by the definition of $s_{*}$ in \eqref{def:time interval}), and show the dynamics of the modulation parameters $\lambda, b$, which can be approximated  by \eqref{eq:ode appsys} up to small error. Define 
\begin{equation}\label{def:modul}
\Mod(s) =\begin{pmatrix}
\frac{\lambda_s}{\lambda}+ b\\
b_s+b^2-\theta\\
1-\gamma_s
\end{pmatrix}.
\end{equation}

\begin{lemma}\label{lem:modul est}
For all $s\in [s_*,s_1]$, then we have
\begin{equation}\label{est:modul}
|\Mod(s)|\lesssim s^{-(K+2)},
\end{equation} 
\begin{equation}\label{est:Pb orth}
|\psld{\eps(s)}{P_b}| <  s^{-(K+2)}.
\end{equation}
\end{lemma}
\begin{proof}	
Since $\eps(s_1)\equiv0$, we may define
$$
s_{**}=\inf\{s\in [s_*,s_1];\   |\psld{\eps(\tau)}{P_b}|<\tau^{-(K+2)}\text{ holds on }  [s,s_1]\}.
$$ 
Therefore, for any  $s\in [s_{**},s_1]$, we have
\begin{equation}\label{est:Pb orth:sm interval}
|\psld{\eps(s)}{P_b}|\leq s^{-(K+2)}.
\end{equation}

Since the estimates of \eqref{est:modul} and \eqref{est:Pb orth} are mixed, we divide the proof into two steps and use the bootstrap argument to conclude the proof.

\noindent{\bf Step 1:} Estimate of the modulation system ${\rm Mod}(s)$ on $[s_{**}, s_1]$.

By \eqref{eq:dnls} and \eqref{eq:Pb}, $\varepsilon$ should satisfy the following difference equation:
\begin{multline}\label{eq:epsilon}
i\eps_s+\partial^2_y \eps+ib\Lambda\eps-\eps+\big(f(P_b+\eps)-f(P_b)\big)+\lambda\, \big(g(P_b+\eps)-g(P_b)\big)\\
-i\left(\frac{\lambda_s}{\lambda}+b\right)\Lambda (P_b+\eps)
+(1-\gamma_s)(P_b+\eps)
+\big(b_s+b^2-\theta\big)\frac{|y|^2}{4}P_b\\
=-\Psi_K e^{-i\frac{b|y|^2}{4}},
\end{multline}
where we used the fact that \eqref{delta:scal}.
 
Formally, by combining the equation \eqref{eq:epsilon} on  $\varepsilon$ with the estimate \eqref{est:error} on $\Psi_K$,  we differentiate in time the orthogonality conditions for $\eps$ provided in Lemma \ref{lem:modul}, to obtain the dynamics of the modulation parameters $\lambda, b$ and $\gamma$. Since it is a standard argument (see e.g. \cite{MeR06:NLS:blp, PlR07:NLS:blp, RaS11:NLS:mini sol}), we only sketch the proof here.

{\bf As for the orthogonality condition $\psld{\eps}{i\Lambda P_b}=0$.} 

Taking time derivative on $\psld{\eps}{i\Lambda P_b}=0$, we obtain 
$$
\dual{\eps_s}{i\Lambda P_b}=- \dual{\eps}{i \partial_s(\Lambda P_b)}.
$$

First,  we consider the contribution from the term $\dual{\eps}{i \partial_s(\Lambda P_b)}$.  By \eqref{def:Pb}, we have
\begin{equation}\label{est:LambdaPb:cal1}
\Lambda P_b=\left(\Lambda P-ib\frac{|y|^2}{2}P\right)e^{-ib\frac{|y|^2}{4}},
\end{equation}
and
\begin{align}\label{est:dsLambPb}
\frac d{d s} ( \Lambda P_b) 
  =  \left( (\Lambda P)_s   - i b_s \frac{|y|^2}{4} \Lambda P- i b_s \frac{|y|^2}{2}P - ib \frac{|y|^2}{2} (P)_s  - b b_s \frac{|y|^4}{8} P \right) e^{-i \frac b4 |y|^2},
\end{align}
where we estimate $(P)_s$ from \eqref{def:P} as follows,
\begin{multline}\label{est:dsP}
( P)_s  =\frac{\lambda_s}{\lambda}\bigg(\sum_{(j,k)\in\Sigma_K}(k+1)b^{2j}\lambda^{k+1}P_{j,k}^+  +i\sum_{(j,k)\in\Sigma_K}(k+1)b^{2j+1}\lambda^{k+1}P_{j,k}^-
\Bigg)\\
+b_s\bigg(\sum_{(j,k)\in\Sigma_K}2jb^{2j-1}\lambda^{k+1}P_{j,k}^+  +i\sum_{(j,k)\in\Sigma_K}(2j+1)b^{2j}\lambda^{k+1}P_{j,k}^-
 \Bigg)\\
=\left(\frac{\lambda_s}{\lambda} + b\right)\bigg(\sum_{(j,k)\in\Sigma_K}(k+1)b^{2j}\lambda^{k+1}P_{j,k}^+  +i\sum_{(j,k)\in\Sigma_K}(k+1)b^{2j+1}\lambda^{k+1}P_{j,k}^-
 \Bigg)\\
 \qquad +\left(b_s+b^2-\theta \right)\bigg(\sum_{(j,k)\in\Sigma_K}2jb^{2j-1}\lambda^{k+1}P_{j,k}^+  +i\sum_{(j,k)\in\Sigma_K}(2j+1)b^{2j}\lambda^{k+1}P_{j,k}^-
 \Bigg)\\
 -  b \bigg(\sum_{(j,k)\in\Sigma_K}(k+1)b^{2j}\lambda^{k+1}P_{j,k}^+  +i\sum_{(j,k)\in\Sigma_K}(k+1)b^{2j+1}\lambda^{k+1}P_{j,k}^-
 \Bigg)\\
- \left(b^2-\theta \right)\bigg(\sum_{(j,k)\in\Sigma_K}2jb^{2j-1}\lambda^{k+1}P_{j,k}^+  +i\sum_{(j,k)\in\Sigma_K}(2j+1)b^{2j}\lambda^{k+1}P_{j,k}^-
 \Bigg),
\end{multline}
and similar estimate about $( \Lambda P)_s $. From the properties of the functions $P_{j,k}^\pm$ in Proposition \ref{prop:profile}, we can deduce from \eqref{est:dsLambPb} that
\begin{equation*}
\sup_{y\in \R}\left(  e^{\frac {|y|}{2}} \left| \frac d{d s} ( \Lambda P_b) (y)\right| \right)\lesssim |\Mod(s)|+b^2(s) + \lambda(s).
\end{equation*}
Thus,  by \eqref{est:small est:rough}, we obtain for any $s\in [s_{**},s_1]$ that
\begin{equation}\label{est:eps orth1:t2}
  |\psld{\eps}{i \partial_s(\Lambda P_b)}| \lesssim \|\eps(s)\|_{2}\Big(|\Mod(s)|+b^2(s) + \lambda(s)\Big)
  \lesssim s^{-2}  |\Mod(s)|+ s^{-(K+2)}.
\end{equation}

Next, we deal with the term $\dual{\eps_s}{i\Lambda P_b}=-\dual{i \varepsilon_s}{\Lambda P_b}$. We firstly estimate the contribution from the first line in \eqref{eq:epsilon}. By \eqref{def:Pb}, \eqref{est:small est:rough},  we have
\begin{equation}\label{est:struct1}
\partial^2_y \eps+ib\Lambda\eps=e^{-ib\frac{|y|^2}{4}}\partial^2_y \left(e^{ib\frac{|y|^2}{4}} \eps\right)+b^2\frac{|y|^2}{4}\eps, 
\end{equation} 
\begin{multline}\label{est:struct2}
f(P_b+\eps)-f(P_b) =e^{-ib\frac{|y|^2}{4}}\left( f\left(P+e^{ib\frac{|y|^2}{4}}\eps\right)-f(P)\right)\\
 =e^{-ib\frac{|y|^2}{4}}df(P)\left(e^{ib\frac{|y|^2}{4}}\eps\right)+O(|\eps|^2)\\
=e^{-ib\frac{|y|^2}{4}}df(P)\left(e^{ib\frac{|y|^2}{4}}\eps\right)+O(s^{-2}|\eps| ),
\end{multline} and
\begin{equation}\label{est:struct3}
\lambda\, \big( g (P_b+\eps)-g(P_b)\big) =  \lambda \mu \delta \varepsilon
= O(s^{-2}|\varepsilon(0)|).
\end{equation}
Therefore,   by \eqref{est:small est:rough}, \eqref{est:Pb orth:sm interval}, and  the definition of $P$ in \eqref{def:P}, we have for any  $s\in [s_{**},s_1]$ that
\begin{multline}\label{est:eps orth1:t1-1}
\dual{-\partial^2_y \eps-ib\Lambda\eps+\eps-(f(P_b+\eps)-f(P_b))-\lambda \left( g (P_b+\eps)-g(P_b)\right)}{\Lambda P_b}\\
=\dual{-\partial^2_y \left(e^{ib\frac{|y|^2}{4}} \eps\right)+e^{ib\frac{|y|^2}{4}}\eps-df(Q) \left(e^{ib\frac{|y|^2}{4}}\eps\right)}{\Lambda Q-ib\frac{|y|^2}{2}Q}+O(s^{-2}\|\eps\|_{H^1})\\
=\dual{L_+ \left(e^{ib\frac{|y|^2}{4}} \eps\right)}{\Lambda Q}-\frac b2\dual{L_-\left(e^{ib\frac{|y|^2}{4}} \eps\right)}{i|y|^2Q}+O(s^{-2}\|\eps\|_{H^1})\\
=\dual{e^{ib\frac{|y|^2}{4}} \eps}{L_+(\Lambda Q)}-\frac b2\dual{e^{ib\frac{|y|^2}{4}} \eps}{iL_-(|y|^2Q)}+O(s^{-2}\|\eps\|_{H^1})\\
=-2\psld{\eps}{ e^{-ib\frac{|y|^2}{4}}Q}+2b\psld{\eps}{ie^{-ib\frac{|y|^2}{4}}\Lambda Q}+O(s^{-2}\|\eps\|_{H^1})\\
=-2\psld{\eps}{ P_b}+2b\psld{\eps}{i \Lambda P_b}+O(s^{-2}\|\eps\|_{H^1})
=O(s^{-(K+2)}),
\end{multline}
where  in the fourth equality we use the algebraic structure  \eqref{fact:algebra} of the operators $L_{\pm}$ and in the last equality we use \eqref{est:Pb orth:sm interval} and the orthogonal relation $\psld{\eps}{i\Lambda P_b}=0$ in  \eqref{eq:orth struct}. 

We secondly consider the contribution from the second line in \eqref{eq:epsilon}. By  the facts that $\psld{P_b}{\Lambda P_b}=0$ and 
\begin{equation}\label{est:basic fact1}
\psld{|y|^2P_b}{\Lambda P_b} = -\norm{y P_b}^2_{2}, 
\end{equation} we have for any $s\in [s_{**}, s_1]$ that
\begin{multline}\label{est:eps orth1:t1-2}
\psld{-i\left(\frac{\lambda_s}{\lambda}+b\right)\Lambda(P_b+\eps)+(1-\gamma_s)(P_b+\eps)+(b_s+b^2-\theta)\frac{|y|^2}{4}P_b}{\Lambda P_b}\\
=-\frac{1}{4}\Big(b_s+b^2-\theta\Big)\norm{yP_b}_2^2+O(|\Mod(s)|\norm{\eps}_{H^1})\\
=-\frac{1}{4}\Big(b_s+b^2-\theta\Big) \norm{yQ}_2^2 +O(s^{-2}|\Mod(s)|),
\end{multline}
where we used \eqref{est:small est:rough}, and the definition of $P_b$ in \eqref{def:Pb}.

Finally, we consider the contribution from the third line in \eqref{eq:epsilon}. By \eqref{est:error}, we have
\begin{equation}\label{est:eps orth1:t1-3}
\left| \psld{\Psi_K}{\Lambda P-ib\frac{|y|^2}{2}P}\right|\lesssim s^{-2} |\Mod(s)| + s^{-2(K+2)}.
\end{equation}
Combining \eqref{est:eps orth1:t2}, \eqref{est:eps orth1:t1-1}, \eqref{est:eps orth1:t1-2} and \eqref{est:eps orth1:t1-3}, we have for any $s\in [s_{**}, s_1]$ that
\begin{equation}\label{est:b dyn}
|b_s+b^2-\theta| \lesssim s^{-2} |\Mod(s)| + s^{-(K+2)}.
\end{equation}

{\bf As for the orthogonality condition $\psld{\eps}{ |y|^2 P_b}=0$.}

Taking time derivative on $\psld{\eps}{ |y|^2 P_b}=0$, we obtain 
$$
\dual{i\eps_s}{i|y|^2 P_b}= \dual{\eps_s}{|y|^2 P_b}=- \dual{\eps}{|y|^2 \partial_s( P_b)}.
$$
By the definition of $P_b$ in \eqref{def:Pb}, we have
\begin{align}\label{est:dsPb}
\frac d{d s} (P_b) 
=  \left( (P)_s   - i b_s \frac{|y|^2}{4} P \right) e^{-i \frac b4 |y|^2}. 
\end{align}
By \eqref{est:dsP} and similar estimate to \eqref{est:eps orth1:t2}, 
we have for any $s\in [s_{**}, s_1]$ that
\begin{equation}\label{est:eps orth2:t2}
|\psld{\eps}{|y|^2 \partial_s( P_b)}| 
\lesssim s^{-2}  |\Mod(s)|+ s^{-(K+2)}.
\end{equation}

Next, we consider the term $\dual{i\eps_s}{i|y|^2 P_b}$. Similar to \eqref{est:eps orth1:t1-1},  we compute by \eqref{def:P}, \eqref{def:Pb}, \eqref{est:small est:rough}  as follows
\begin{multline}
\quad \dual{-\partial^2_y \eps-  ib\Lambda\eps+\eps-(f(P_b+\eps)-f(P_b))-\lambda \left( g (P_b+\eps)-g(P_b)\right)}{i |y|^2 P_b}  \\
=\dual{-\partial^2_y \left(e^{ib\frac{|y|^2}{4}} \eps\right)+e^{ib\frac{|y|^2}{4}}\eps-df(Q) \left(e^{ib\frac{|y|^2}{4}}\eps\right)}{i|y|^2Q}+O(s^{-2}\|\eps\|_{H^1}) \\
 =\dual{L_-\left(e^{ib\frac{|y|^2}{4}} \eps\right)}{i|y|^2Q}+O(s^{-2}\|\eps\|_{H^1})  \\
 =\dual{e^{ib\frac{|y|^2}{4}} \eps}{iL_-(|y|^2Q)}+O(s^{-2}\|\eps\|_{H^1})  \\
 =-2\psld{\eps}{ie^{-ib\frac{|y|^2}{4}}\Lambda Q}+O(s^{-2}\|\eps\|_{H^1})  \\
 =-2\psld{\eps}{ie^{-ib\frac{|y|^2}{4}}\Lambda P}+O(s^{-2}\|\eps\|_{H^1}) \\
=-2\psld{\eps}{i \Lambda P_b} + b \psld{\varepsilon}{|y|^2P_b}+O(s^{-2}\|\eps\|_{H^1})
=O(s^{-(K+2)}), \label{est:eps orth2:t1-1} 
\end{multline}
where we used \eqref{fact:algebra} in the fourth equality, and the orthogonal relation \eqref{eq:orth struct} in the last equality.  Similar to \ \eqref{est:eps orth1:t1-2}, we have

\begin{multline}\label{est:eps orth2:t1-2}
\psld{-i\left(\frac{\lambda_s}{\lambda}+b\right)\Lambda\big(P_b+\eps\big)+(1-\gamma_s)\big(P_b+\eps\big)+\big(b_s+b^2-\theta\big)\frac{|y|^2}{4}P_b}{i |y|^2 P_b}\\
=-\left(\frac{\lambda_s}{\lambda}+b\right)\psld{\Lambda P_b}{|y|^2P_b}+O(|\Mod(s)|\norm{\eps}_{H^1})\\
=\left(\frac{\lambda_s}{\lambda}+b\right) \norm{yQ}_2^2 +O(s^{-2}|\Mod(s)|),
\end{multline}
where we used \eqref{est:basic fact1} in the first equality. By \eqref{est:error}, we have

\begin{equation}\label{est:eps orth2:t1-3}
\left| \psld{\Psi_K}{i |y|^2 P}\right|\lesssim s^{-2} |\Mod(s)| + s^{-2(K+2)}.
\end{equation}
By combining \eqref{est:eps orth2:t2}, \eqref{est:eps orth2:t1-1}, \eqref{est:eps orth2:t1-2} and \eqref{est:eps orth2:t1-3}, we have for any $s\in [s_{**}, s_1]$ that
\begin{equation}\label{est:lambda dyn}
\left|\frac{\lambda_s}{\lambda}+b\right| \lesssim s^{-2} |\Mod(s)| + s^{-(K+2)}.
\end{equation}

{\bf As for the orthogonality condition $\psld{\eps}{i \rho_b}=0$.}

Taking time derivative on $\psld{\eps}{i \rho_b}=0$, we obtain 
$$
-\dual{i\eps_s}{\rho_b}= \dual{\eps_s}{i\rho_b}=- \dual{\eps}{ \partial_s( \rho_b)}.
$$
By the definition of $\rho_b$ in \eqref{def:rhob}, we have
\begin{align}
\frac d{d s} (\rho_b) 
=    - i b_s \frac{|y|^2}{4} \rho \  e^{-i \frac b4 |y|^2}. 
\end{align}
By similar estimate as \eqref{est:eps orth1:t2}, 
we have for any $s\in [s_{**}, s_1]$ that
\begin{equation}\label{est:eps orth3:t2}
|\psld{\eps}{ \partial_s( \rho_b)}| 
\lesssim s^{-2}  |\Mod(s)|+ s^{-(K+2)}.
\end{equation}

Next, we consider the term $\dual{i\eps_s}{ \rho_b}$. Similar to  \eqref{est:eps orth1:t1-1}, we computer by \eqref{est:small est:rough} as follows

\begin{align}
& \quad \dual{-\partial^2_y \eps-ib\Lambda\eps+\eps-(f(P_b+\eps)-f(P_b))-\lambda \left( g (P_b+\eps)-g(P_b)\right)}{ \rho_b} \nonumber\\
&=\dual{-\partial^2_y \left(e^{ib\frac{|y|^2}{4}} \eps\right)+e^{ib\frac{|y|^2}{4}}\eps-df(Q) \left(e^{ib\frac{|y|^2}{4}}\eps\right)}{\rho}+O(s^{-2}\|\eps\|_{H^1}) \nonumber\\
&=\dual{L_+\left(e^{ib\frac{|y|^2}{4}} \eps\right)}{\rho}+O(s^{-2}\|\eps\|_{H^1}) \nonumber\\
&=\dual{e^{ib\frac{|y|^2}{4}} \eps}{L_+ \rho}+O(s^{-2}\|\eps\|_{H^1}) \nonumber\\
&=\psld{\eps}{e^{-ib\frac{|y|^2}{4}} |y|^2 Q}+O(s^{-2}\|\eps\|_{H^1}) \nonumber\\
& =  \psld{\varepsilon}{|y|^2P_b}+O(s^{-2}\|\eps\|_{H^1})
=O(s^{-(K+2)}),\label{est:eps orth3:t1-1}
\end{align}
where we used \eqref{fact:algebra} in the fourth equality,  and the orthogonal relation  \eqref{eq:orth struct} in the last equality. Similar to \eqref{est:eps orth1:t1-2}, we have

\begin{multline}\label{est:eps orth3:t1-2}
\psld{-i\left(\frac{\lambda_s}{\lambda}+b\right)\Lambda(P_b+\eps)+(1-\gamma_s)(P_b+\eps)+(b_s+b^2-\theta)\frac{|y|^2}{4}P_b}{ \rho_b}\\
=\left(1-\gamma_s\right)\psld{ P_b}{\rho_b}+O\left(s^{-2}|\Mod(s)|\right)+ O\left(s^{-(K+2)}\right) \\
=\left(1-\gamma_s\right)\psld{ Q}{\rho} +O(s^{-2}|\Mod(s)|)+ O\left(s^{-(K+2)}\right)\\
= \frac12 \left(1-\gamma_s\right) \norm{yQ}^2_2 +O(s^{-2}|\Mod(s)|)+ O\left(s^{-(K+2)}\right),
\end{multline}
where we used \eqref{est:b dyn},  \eqref{est:lambda dyn} in the first equality, and \eqref{est:basic fact2} in the last equality. Last, by \eqref{est:error}, we have

\begin{equation}\label{est:eps orth3:t1-3}
\left| \psld{\Psi_K}{\rho}\right|\lesssim s^{-2} |\Mod(s)| + s^{-2(K+2)}.
\end{equation}
By combining \eqref{est:eps orth3:t2}, \eqref{est:eps orth3:t1-1}, \eqref{est:eps orth3:t1-2} and \eqref{est:eps orth3:t1-3}, we have for any $s\in [s_{**}, s_1]$ that
\begin{equation}\label{est:gam dyn}
\left|1-\gamma_s\right| \lesssim s^{-2} |\Mod(s)| + s^{-(K+2)}.
\end{equation}

By \eqref{est:b dyn}, \eqref{est:lambda dyn} and \eqref{est:gam dyn}, we obtain that
\begin{equation*}
 |\Mod(s)|
\lesssim  s^{-2} |\Mod(s)|  + s^{-(K+2)}, 
\end{equation*}
which implies for any $s\in [s_{**}, s_1]$ that
\begin{equation}\label{est:modul2}
|\Mod(s)|\lesssim s^{-(K+2)}.
\end{equation}

\noindent {\bf Step 2: } Estimate of $\psld{\varepsilon}{P_b}$ on $[s_{**}, s_1]$.
 
On the one hand, by the mass conservation and \eqref{eq:final data}, we have
\begin{equation*}
\norm{u(s)}_2^2=\norm{u(s_1)}_2^2=\norm{P_b(s_1)}_2^2.
\end{equation*}
On the other hand, by the modulation decomposition \eqref{eq:refined modul},  we have
\begin{equation*}
\psld{\eps(s)}{P_b}=\frac12\left(\norm{u(s)}_2^2-\norm{P_b(s)}_2^2-\norm{\eps(s)}_2^2\right)=
-\frac 12 \norm{\eps(s)}_2^2+\frac 12 \Big( \norm{P_b(s_1)}_2^2-\norm{P_b(s)}_2^2\Big).
\end{equation*}
Moreover, by  \eqref{est:Pblamb dmass}, \eqref{est:small est:rough} and \eqref{est:modul2}, we compute the later as follows
\begin{equation*}
\frac {d}{ds}\int_{\R} |P_b|^2 \, dy
\lesssim \lambda(s) \left| \rm Mod(s)\right| + s^{-2(K+2)} \lesssim s^{-(K+4)}. 
\end{equation*}
By integrating over $[s, s_1]$ and combining \eqref{est:small est:rough},  we obtain for any $s\in [s_{**},s_1]$ that 
\begin{equation*}
|\psld{\eps(s)}{P_b}|\lesssim s^{-2K}+ s^{-(K+3)} \lesssim  s^{-(K+3)}.
\end{equation*}
Therefore, we obtain $s_{**}=s_*$ for sufficiently large $s_*$,  and the estimates \eqref{est:modul} and \eqref{est:Pb orth} are proved on 
$[s_*,s_1]$.
\end{proof}

\begin{corollary}\label{cor:Q orth}
	For all $s\in [s_*,s_1]$, then we have
	\begin{equation}\label{est:Q orth}
	|\psld{\eps(s)}{Q}|\lesssim  s^{-(K+1)}.
	\end{equation}
	\end{corollary}
\begin{proof} Note that 
	\begin{equation*}
	\psld{\eps(s)}{Q}=\psld{\eps(s)}{P_b}-\psld{\eps(s)}{P_b-Q}.
	\end{equation*}
	By  \eqref{est:small est:rough}, \eqref{est:Pb orth} and the fact that
	 $|P_b - Q||\lesssim  Q^{\frac 12}\big( |b| + \lambda \big)  \lesssim  Q^{\frac 12} s^{-1}$, we can obtain the result.
	\end{proof}

\subsection{The monotonicity formula: the  Energy-Morawetz estimate }\label{sect:Mono form}After we obtain the dynamics of the modulation parameters in last subsection. Now, we turn to introduce the Energy-Morawetz functional in this subsection, use it in next subsection to improve the estimates about the remainder $\varepsilon$ and modulation parameters $\lambda$, $b$, and then close the bootstrap argument. 

The Energy-Morawetz estimate was firstly introduced to construct minimal mass blow-up solution of the inhomogeneous NLS in \cite{RaS11:NLS:mini sol}, see also \cite{KrLR13:HalfW:nondis, LeMR:CNLS:blp, MaP17:BO:mini sol, MeRS14:NLS:blp}), and recently, it was also successfully applied by B.~Dodson to obtain the global well-posedness and scattering results of the defocusing, nonlinear wave equation in the critical Sobolev space $\dot H^s(\R^3)$ for $\frac12\leq s<1$ in \cite{Dod:NLW:sct1, Dod:NLW:sct2}.

First, we recall the well-known coercivity property of the linearized operators $L_{\pm}$ around $Q$ in $H^1_{\rm rad}(\R)$.
\begin{lemma}\label{lem:Lpm coer}
For any $\eps=\eps_1+i\eps_2\in H^1_{\rm rad}(\R)$, then there exists constant $C>0$ such that  
\begin{equation*}\label{est:Lpm coer}
\dual{L_+\eps_1}{\eps_1}+\dual{L_-\eps_2}{\eps_2}\geq  C \norm{\eps}_{H^1}^2-\frac 1C\left( \psld{\eps_1}{Q}^2+\psld{\eps_1}{|y|^2Q}^2+\psld{\eps_2}{\rho}^2\right).
\end{equation*}
\end{lemma}
\begin{proof}
Please refer to Lemma $3.2$ in \cite{RaS11:NLS:mini sol} for more details (See also \cite{MeR04:NLS:blp, MeR05:NLS:blp, MeR06:NLS:blp, We86:NLS:stab}).
	\end{proof}

Now, we define the  linearized energy functional of the remainder $\eps$ according to \eqref{eq:epsilon} as follows
\begin{multline*}
H(s,\eps):=\frac12\norm{\partial_y \eps}_2^2+\frac12\norm{\eps}_2^2-\int_{\R} \Big(F(P_b+\eps)-F(P_b)-dF(P_b)\eps\Big)dy\\
-\lambda \int_{\R} \Big(G(P_b+\eps)-G(P_b)-dG(P_b)\eps\Big)dy.
\end{multline*}
Note that  the equation \eqref{eq:epsilon} can be rewrited as 
\begin{equation}\label{eq:eps v2}
i\eps_s-{D_\eps} H(s,\eps)+\ModOp(s)P_b-i\frac{\lambda_s}{\lambda} \Lambda\eps
+(1-\gamma_s)\eps
+e^{-ib\frac{|y|^2}{4}}\Psi_K=0,
\end{equation}
where  $D_\eps H(s,\eps)$ denotes the Fr\'echet derivative of the functional $H(s,\eps)$ with respect to $\eps$ and
\begin{equation*}
\ModOp(s)P_b =-i\left(\frac{\lambda_s}{\lambda}+b\right)\Lambda P_b
+(1-\gamma_s)P_b
+\big(b_s+b^2-\theta\big)\frac{|y|^2}{4}P_b.
\end{equation*}

First, as the consequence of Lemma \ref{lem:Lpm coer}, we have the coercivity property for the energy functional $H(s,\varepsilon)$ under the orthogonality conditions of $\varepsilon$ (see \eqref{eq:orth struct}, \eqref{est:Q orth}) as follows.

\begin{lemma}\label{lem:H coer}
	For all $s\in[s_*,s_1]$,  then we have
\begin{equation*}
	H(s,\eps)\gtrsim 
	\norm{\eps}_{H^1}^2+O(s^{-2(K+1)}).
\end{equation*}
\end{lemma}

\begin{proof}
Firstly, from  \eqref{eq:orth struct}, \eqref{est:small est:rough} and \eqref{est:Q orth},  the following estimates hold:
	\begin{align*}
	\psld{\eps}{|y|^2 Q}&=\psld{\eps}{|y|^2P_b}+O(|b| \norm{\eps}_2)+O(\lambda \norm{\eps}_2) =O(s^{-1}\norm{\eps}_{H^1}) ,\\
	\psld{\eps}{i\rho}  &  =\psld{\eps}{i\rho_b} +O(|b|\norm{\eps}_2)=O(s^{-1}\norm{\eps}_{H^1}),\\
	\psld{\eps}{ Q}&        =O(s^{-(K+1)}),
	\end{align*}
where we used the fact that  $|P_b - Q||\lesssim  Q^{\frac 12}\big( |b| + \lambda \big)$ in the first equality.

Next, If we denote $\eps = \eps_1+ i\eps_2$, then 
	$$\left|  F(P_b+\eps)-F(P_b)-dF(P_b)\eps 
	- \frac{1}{2} 5 Q^{4} \eps_1^2 - \frac 12 Q^{4} \eps_2^2\right|
	\lesssim e^{-\frac 12 |y|} |\eps|^3 +|\eps|^{6}+ |\eps|^2 (|b|+\lambda).
	$$
	Thus, from \eqref{est:small est:rough},  we have
	$$\left| \int_{\R} \left(F(P_b+\eps)-F(P_b)-dF(P_b)\eps 
	-\frac{1}{2} 5Q^{4} \eps_1^2 - \frac 12Q^{4} \eps_2^2\right)\, dy\right|
	\lesssim 
	O(s^{-1}\norm{\eps}_{H^1}^2),
	$$
	and
	\begin{equation*}
	\lambda \int_{\R} \Big(G(P_b+\eps)-G(P_b)-dG(P_b)\eps\Big)dx= O(s^{-2}\norm{\eps}_{H^1}^2).
	\end{equation*}
Therefore, we have
	$$
	\left|H(s,\eps) - \frac 12 \dual{L_+\eps_1}{\eps_1}-\frac 12 \dual{L_-\eps_2}{\eps_2}\right|\lesssim 
	O(s^{-1}\norm{\eps}_{H^1}^2), 
	$$
which together with the coercivity properties of   $L_{\pm}$ in Lemma \ref{lem:Lpm coer} implies the result.
\end{proof}

For future reference, we also need the following localized coercivity property.
\begin{lemma}\label{lem:locH:coer}
	There exists $A_0>1$ such that for any $A>A_0$, we have
	\begin{equation*}
	\frac12\int_{\R}\phantom{}  \partial^2_y\phi_A \cdot |\partial_y \eps|^2\, dy+\frac12\norm{\eps}_2^2-\int_{\R} \Big(F(P_b+\eps)-F(P_b)-dF(P_b)\eps\Big)\, dy\gtrsim \norm{\eps}^2_{2}+O(s^{-2(K+1)}),
	\end{equation*}
	where the localized function $\phi_A$ is defined in Section \ref{sect:notation},
\end{lemma}

\begin{proof}
	 Please refer to \cite{RaS11:NLS:mini sol} for details.
\end{proof}

Second, we compute the time variation of the linearized energy $H(s,\varepsilon)$  as follows.

\begin{lemma}\label{lem:dsH} For all $s\in[s_*,s_1]$,  we have
	\begin{equation}\label{eq:H:ds}
	\frac{d}{d s}\Big(H(s,\eps(s))\Big)=
	\frac{\lambda_s}{\lambda} \Big( \norm{\partial_y \eps}_2^2-\dual{f(P_b+\eps)-f(P_b)}{\Lambda\eps} \Big)
	+O\big(s^{-(2K+3)}\big)+O(s^{-2}\norm{\eps}_{H^1}^2).
	\end{equation}
\end{lemma}

\begin{proof}
We compute the  time derivative for $H$ as follows.
	\begin{equation}\label{eq:dsH}
	\frac{d}{d s}\Big(H(s,\eps(s))\Big)= D_s H(s,\eps) + \dual{D_\eps H(s,\eps)}{\eps_s},
	\end{equation}
	where $D_s H(s,\eps) $ denotes differentiation of the functional $H(s,\eps) $ with respect to $s$. Therefore, we have
	\begin{align*}
	D_s H(s,\eps)  =  &  - \int_{\R} (P_b)_s \Big( f(P_b+\eps) - f(P_b) - df(P_b)\eps \Big) \, dy
	-\lambda \int_{\R} (P_b)_s \Big( g(P_b+\eps)-g(P_b) - dg(P_b)\eps\Big) \, dy\\
	&- \lambda_s \ \int_{\R} \Big( G(P_b+\eps) - G(P_b) - dG(P_b)\eps\Big)\, dy.
	\end{align*}
	Note that 
	$$e^{ i \frac {b |y|^2}{4}} (P_b)_s  =  P_s  - i b_s \frac {|y|^2}{4} P 
	=  P_s - i \left(b_s +b^2 - \theta \right) \frac {|y|^2}{4} P +i\left(b^2 - \theta \right) \frac {|y|^2}{4} P.
	$$
	By \eqref{est:small est:rough}, \eqref{est:dsP} and Lemma \ref{est:modul}, we obtain
	$$
	|(P_b)_s |  \lesssim \Big(  \big|{\rm Mod(s)}\big| + b^2 + \lambda \Big ) Q^{\frac 12} \lesssim s^{-2} e^{-\frac{|y|}{2}},  $$
and 
	$$ \left|\lambda_s \right| =|\lambda|   \left|\frac {\lambda_s}{\lambda}\right| \lesssim|\lambda|  \Big( \big|{\rm Mod(s) }\big|+ |b| \Big)   \lesssim  s^{-3}.
	$$
	Thus, we have 
	\begin{equation}\label{est:dsH1}
	| D_s H(s,\eps) | \lesssim s^{-2} \|\eps\|_{H^1}^2.
	\end{equation}

	Now, we compute $\dual{D_\eps H(s,\eps)}{\eps_s}$.	By \eqref{eq:eps v2} and the fact that $\dual{i{D_\eps} H(s,\eps)}{{D_\eps}H(s,\eps)}=0$, we have
	\begin{multline}\label{eq:depsH}
	\dual{{D_\eps} H(s,\eps)}{\eps_s}
	=
	\dual{{D_\eps} H(s,\eps)}{i \ModOp(s)P_b}+
	\frac{\lambda_s}{\lambda} \dual{{D_\eps} H(s,\eps)}{\Lambda\eps}\\
	+(1-\gamma_s)\dual{{D_\eps} H(s,\eps)}{i \eps}
	+ \dual{{D_\eps} H(s,\eps)}{i\Psi_K  e^{-ib\frac{|y|^2}{4}}}.
	\end{multline}
By similar estimates as those in \eqref{est:struct1} \eqref{est:struct2},  \eqref{est:struct3} and  \eqref{est:eps orth1:t1-1}, we have
	\begin{multline*}
	{D_\eps} H(s,\eps)
	=
	-\partial^2_y \eps+\eps-\big(f(P_b+\eps)-f(P_b)\big)-\lambda \big(g(P_b+\eps)-g(P_b)\big)
	\\
	=e^{-ib\frac{|y|^2}{4}}\left(\big(-\partial^2_y+1-df(Q)\big)\left(e^{ib\frac{|y|^2}{4}}\eps\right) \right)+ib\Lambda\eps+b^2\frac{|y|^2}{4}\eps+ O(s^{-2} |\eps|).
	\end{multline*}
	Therefore, by the orthogonal relations in \eqref{eq:orth struct}, \eqref{est:Pb orth} and the similar estimate as in \eqref{est:eps orth1:t1-1}, we have
	\begin{align}
	\dual{{D_\eps} H(s,\eps)}{\Lambda P_b} & \nonumber\\
	= & \dual{\big(-\partial^2_y+1-df(Q)\big)\left(e^{ib\frac{|y|^2}{4}}\eps\right)}{\Lambda Q-ib\frac{|y|^2}{2}Q} + b \dual{i\Lambda\eps}{\Lambda P_b} + O(s^{-2}\norm{\eps}_2) \nonumber\\
	=& -2\psld{\eps}{ P_b}+2b\psld{\eps}{i \Lambda P_b} + b \dual{i\Lambda\eps}{\Lambda P_b}  + O(s^{-2}\norm{\eps}_2) \nonumber
	\\
	=& -2\psld{\eps}{P_b}+2b\psld{\eps}{i\Lambda P_b}+ O(|b|\norm{\varepsilon}_{H^1})+  O(s^{-2}\norm{\eps}_2)=O(s^{-(K+1)}), \label{est:d Mod1}
	\end{align}
and
	\begin{multline}\label{est:d Mod2}
	\dual{{D_\eps} H(s,\eps)}{ iP_b} \\
	= \dual{\big(-\partial^2_y+1-df(Q)\big)\left(e^{ib\frac{|y|^2}{4}}\eps\right)}{iQ} + b \dual{i\Lambda\eps}{i P_b} + O(s^{-2}\norm{\eps}_2) \\
	=\dual{L_-\left(e^{ib\frac{|y|^2}{4}}\eps\right)}{i Q} - b \dual{\eps}{\Lambda P_b} +  O(s^{-2}\norm{\eps}_2)
	\\
	=O(|b|\norm{\varepsilon}_2)+O(s^{-2}\norm{\eps}_2)=O(s^{-(K+1)}),
	\end{multline}
	and  
	\begin{multline}\label{est:d Mod3}
	\dual{{D_\eps} H(s,\eps)}{ i\frac{|y|^2}{4}P_b}\\
	=\dual{\big(-\partial^2_y+1-df(Q)\big)\left(e^{ib\frac{|y|^2}{4}}\eps\right)}{i\frac{|y|^2}{4}Q} + b \dual{i\Lambda\eps}{i \frac{|y|^2}{4} P_b} + O(s^{-2}\norm{\eps}_2) \\
	= \dual{L_-\left(e^{ib\frac{|y|^2}{4}}\eps\right)}{i\frac{|y|^2}{4}Q} + O(|b|\norm{\varepsilon}_2) + O(s^{-2}\norm{\eps}_2) 
\\	=-\psld{\eps}{i \Lambda P_b}+O(s^{-1}\norm{\eps}_2)=O(s^{-(K+1)}).
	\end{multline}
By combining \eqref{est:d Mod1}, \eqref{est:d Mod2}, \eqref{est:d Mod3} 	with Lemma \ref{lem:modul est}, we obtain
\begin{equation}\label{est:d Mod}
	\dual{{D_\eps} H(s,\eps)}{i \ModOp(s)P_b}=O(s^{-(2K+3)}).
	\end{equation}

	Next, we have
	\begin{equation}\label{eq:d leps}
\dual{{D_\eps} H(s,\eps)}{\Lambda\eps}=
	\dual{-\partial^2_y \eps+\eps-\Big(f(P_b+\eps)-f(P_b)\Big)-\lambda  \Big(g(P_b+\eps)-g(P_b)\Big)}{\Lambda\eps}.
	\end{equation}
By \eqref{est:small est:rough} and simple computations, we have
	\begin{equation*}
	\dual{-\partial^2_y \eps}{\Lambda\eps}=\norm{\partial_y \eps}_2^2,\quad \dual{ \eps}{\Lambda\eps}=0,
	\end{equation*}
	and 
	$$
	\left|\dual{\lambda \big(g(P_b+\eps)-g(P_b)\big)}{\Lambda\eps}\right|\lesssim O(s^{-2}\norm{\eps}^2_{H^1}).
	$$
	Thus,
\begin{equation}\label{est:d leps}
	\frac{\lambda_s}{\lambda} \dual{{D_\eps} H(s,\eps)}{\Lambda\eps}=
	\frac{\lambda_s}{\lambda}\Big(\norm{\partial_y \eps}_2^2-\dual{f(P_b+\eps)-f(P_b)}{\Lambda\eps}\Big)+O\big(s^{-3} \|\eps\|_{H^1}^2\big).
	\end{equation}
	
	For the third term in the right-hand side of \eqref{eq:depsH}, we have
	\begin{multline}\label{est:d esp}
	|(1-\gamma_s)\dual{{D_\eps} H(s,\eps)}{i\eps}|=\Big|(1-\gamma_s)\dual{\Big(f(P_b+\eps)-f(P_b)\Big)+\lambda \Big(g(P_b+\eps)-g(P_b)\Big)}{\eps}\Big|\\
	\lesssim
	|\Mod(s)|\Big(
	\norm{\eps}_2^2+\norm{\eps}_{H^1}^{6}\Big)=O(s^{-4}\norm{\eps}_{H^1}^2).
	\end{multline}
	
	Finally, by  \eqref{est:error} , \eqref{est:small est:rough} and Lemma~\ref{lem:modul est},  the fourth term in the right-hand side of  \eqref{eq:depsH}
	is estimated by  
	\begin{multline}\label{est:d psi}
	\left|\dual{{D_\eps} H(s,\eps)}{i\Psi_K e^{-ib\frac{|y|^2}{4}}}\right| 
	\lesssim \left(\norm{\varepsilon}_{H^1}+ \norm{\varepsilon}_{H^1}^5\right) \Big(\lambda \big|{\rm Mod(s)}\big| + (b^2 + \lambda)^{K+2}\Big)
	\\
	\leq O(s^{-(K+4)}\norm{\eps}_{H^1})
	\leq O(s^{-(2K+3)})+O(s^{-5}\norm{\eps}_{H^1}^2).
	\end{multline}
	
By  \eqref{eq:dsH},  \eqref{est:dsH1}, \eqref{eq:depsH},  \eqref{est:d Mod}, \eqref{eq:d leps}, \eqref{est:d leps}, \eqref{est:d esp} and \eqref{est:d psi},  we can complete the proof of Lemma.
\end{proof}

Note that as in \cite{RaS11:NLS:mini sol}, the time derivative of the linearized energy functional $H(s,\varepsilon)$ in \eqref{eq:H:ds} cannot be well controlled because of the lack of the good estimate about the additional term $\dual{f(P_b+\eps)-f(P_b)}{\Lambda\eps}$, and we need introduce a Morawetz type functional  such as 
$$
\frac 12\, \Im\int_{\R} \partial_y \left(\frac{|y|^2}{2}\right) \partial_y \eps \bar\eps\, dy
$$ in $H(s,\varepsilon)$ to cancel the effect from $\dual{f(P_b+\eps)-f(P_b)}{\Lambda\eps}$.
In fact, we need use a localized function to replace $ |y|^2/2$ due to the lack of control on $\norm{y\eps}_2$. 

From now on, we choose $A>A_0$  for sufficiently large $A_0$. We define the localized Morawetz functional $J(\eps)$ by
\begin{equation*}
J(\eps)=\frac12\, \Im\int_{\R}\partial_y \phi_A (y)\cdot \partial_y \varepsilon\, \bar\varepsilon \, dy,
\end{equation*}
where the localized function $\phi_A$ is defined in Section \ref{sect:notation}, and  denote the Energy-Morawetz functional $S(s, \eps)$  by 
\begin{equation*}
S(s,\eps)=\frac{1}{\lambda^{4}(s)}\Big(H(s,\eps(s))+b(s)\cdot J(\eps(s))\Big).
\end{equation*}

  First, we have 

\begin{proposition}\label{prop:S:coer}
	For any $s\in[s_*,s_1]$, then we have 
	\begin{equation*}
	S(s,\eps(s))\gtrsim \frac{1}{\lambda^{4}(s)}\left(
	\norm{\eps(s)}_{H^1}^2+O(s^{-2(K+1)}\right).
	\end{equation*}
\end{proposition}

\begin{proof}
By \eqref{est:small est:rough}, we have
	\begin{equation*}
	\big|b\cdot J(\eps)\big|\leq |b|\norm{\partial_y \phi_A}_{\infty}\norm{\eps}^2_{H^1}\lesssim O(s^{-1}\norm{\eps}_{H^1}^2),
	\end{equation*}
which together with Lemma \ref{lem:H coer} implies the result.
\end{proof}

Before we compute the time derivative of functional $S(s,\varepsilon)$, we deal with the time derivative of functional $J(\varepsilon)$.

\begin{lemma}\label{lem:dsJ}
	For all $s\in[s_*,s_1]$,  then we have
	\begin{multline*}
	\frac{d}{d s}[J(\eps(s))]=\int_{\R}\phantom{}   \partial^2_y\phi_A\cdot | \partial_y \eps|^2 \, dy
	-\frac14\int_{\R}\partial^4_y \phi_A |\eps|^2 \, dy\\
	-\dual{f(P_b+\eps)-f(P_b)}{\frac12\partial^2_y\phi_A\eps+\partial_y\phi_A\partial_y \eps}
	+O(s^{-(2K+2)})+O(s^{-1}\norm{\eps}_{H^1}^2).
	\end{multline*}
\end{lemma}

Note that the term $\dual{f(P_b+\eps)-f(P_b)}{\frac12\partial^2_y\phi_A\eps+\partial_y\phi_A\partial_y \eps}$ in the above equality is the localized version of $ \dual{f(P_b+\eps)-f(P_b)}{ \Lambda \eps}$ in the right hand side of \eqref{eq:H:ds}.

\begin{proof}
	By integration by parts, we have
	\begin{equation*}
	\frac{d}{d s}[J(\eps(s))]=\Re\int_{\R}i\eps_s \left(\frac12\partial^2_y\phi_A\bar\eps+\partial_y\phi_A\partial_y \bar\eps\right)\, dy.
	\end{equation*}
	
By \eqref{est:small est:rough}, we estimate the term from the first line of \eqref{eq:epsilon} as follows. 
	\begin{gather*}
	\Re\int_{\R}-\partial^2_y\eps\left(\frac12\partial^2_y \phi_A\bar\eps+\partial_y\phi_A\partial_y \bar\eps\right)dy=\int_{\R^d}\phantom{}    \partial^2_y\phi_A|\partial_y \eps|^2\,  dy-\frac14\int_{\R}\partial^4_y \phi_A|\eps|^2 \, dy,
	\\
	\Re\int_{\R}\eps\left(\frac12\partial^2_y \phi_A\bar\eps+\partial_y\phi_A\partial_y \bar\eps\right)dy=0,
	\\
	b\cdot \Re\int_{\R}i\Lambda\eps\left(\frac12\partial^2_y\phi_A\bar\eps+\partial_y \phi_A\partial_y \bar\eps\right)dy=O \left(|b| \norm{\eps}^2_{H^1}\right) = O \left(s^{-1}\norm{\eps}^2_{H^1}\right),
	\end{gather*}
and 
\begin{equation*}
	\lambda \cdot \Re\int_{\R}\Big(g(P_b+\eps)-g(P_b)\Big) \left(\frac12\partial^2_y\phi_A\bar\eps+\partial_y\phi_A\partial_y \bar\eps\right)dy=O(\lambda \norm{\eps}_{H^1}^2)=O(s^{-(2K+2)}).
	\end{equation*}
	
For  the term from the second line of \eqref{eq:epsilon}, we obtain from \eqref{est:small est:rough} and Lemma \ref{lem:modul est} that 
	\begin{multline*}
	\left|\dual{i\left(\frac{\lambda_s}{\lambda}+b\right)\Lambda (P_b+\eps)
		-(1-\gamma_s)(P_b+\eps)
		-(b_s+b^2-\theta)\frac{|y|^2}{4}P_b}
	{\frac12\partial^2_y\phi_A \eps+\partial_y\phi_A\partial_y \eps}\right|\\
	\lesssim |\Mod(s)| \norm{\eps}_{H^1}\lesssim O(s^{-(2K+2)}).
	\end{multline*}
	
	Finally, by  \eqref{est:error} and Lemma~\ref{lem:modul est}, we have
	\begin{multline*}
	\left| \dual{\Psi_K e^{-i \frac{b |y|^2}{4}}}{\frac12\partial^2_y\phi_A\bar\eps+\partial_y\phi_A\partial_y \bar\eps}\right| \\
	\lesssim   \Big(\lambda \big|{\rm Mod(s)}\big| + (b^2 + \lambda)^{K+2}\Big)\norm{\eps}_{H^1}
	\\
	\leq O(s^{-(K+4)}\norm{\eps}_{H^1})
	\leq O(s^{-(2K+4)}).
	\end{multline*}

Together with the above estimates, we can obtain the result.
\end{proof}

\begin{proposition}\label{prop:dsS}
	For any $s\in[s_*,s_1]$, we have
	\begin{equation*}
	\frac{d}{d s}\Big( S(s,\eps(s))\Big) \gtrsim 
	\frac{b}{\lambda^{4}(s)}\left(\norm{\eps(s)}_{H^1}^2+O\big(s^{-2(K+1)}\big)\right).
	\end{equation*}
\end{proposition}

\begin{proof}
	By the definition of $S$, we have
	\begin{equation*}
	\frac{d}{d s}\Big(S(s,\eps(s))\Big)=\frac{1}{\lambda^{4}}\left( -4\frac{\lambda_s}{\lambda} \big( H(s,\eps)+b\cdot J(\eps) \big) + 
	\frac{d}{d s} \big(H(s,\eps(s))\big) 
	+b\cdot \frac{d}{d s}\big(J(\eps(s))\big)+ b_s\cdot J(\eps) \right). 
	\end{equation*}
	
	First, we claim the following estimate
	\begin{equation}\label{est:dsHJ}
	\frac{d}{d s} \big( H(s,\eps(s))\big) 
	+b\cdot \frac{d}{d s}\big(J(\eps(s))\big)=-b\norm{\partial_y\eps}_2^2+ 
	b\int_{\R} \partial^2_y\phi_A|\partial_y\eps|^2 \, dy+
O\left(	\frac{b}{A}\norm{\eps}_{H^1}^2\right)+  O(s^{-(2K+3)}) .
	\end{equation}
	
	 Indeed, by integration by parts, we have
	\begin{multline*}
	-\Re\int_{\R} \Big(f(P_b+\eps)-f(P_b)\Big) \cdot  \Lambda \bar\eps \, dy\\
	=
	-\frac 12\Re\int_{\R} \Big(f(P_b+\eps)-f(P_b)\Big) \bar\eps\, dy
	-\Re\int_{\R} y\partial_y \Big(F(P_b+\eps)-F(P_b)-dF(P_b)\eps\Big)\,
	dy\\+
	\Re\int_{\R} \Big(f(P_b+\eps)-f(P_b)-df(P_b)\eps\Big)y\partial_y\bar P_b\, dy
	\\
	=
	-\frac 12\Re\int_{\R} \Big(f(P_b+\eps)-f(P_b)\Big) \bar\eps \, dy
	+\Re\int_{\R}\Big(F(P_b+\eps)-F(P_b)-dF(P_b)\eps\Big)\,
	dy\\+
	\Re\int_{\R} \Big(f(P_b+\eps)-f(P_b)-df(P_b)\eps\Big)y\partial_y \bar P_b\, dy,
	\end{multline*}
	and
	\begin{multline*}
	-\Re\int_{\R} \Big(f(P_b+\eps)-f(P_b)\Big)\cdot  \left(\frac12\partial^2_y\phi_A\bar\eps+\partial_y \phi_A\partial_y \bar\eps\right)\, dy\\
	=
	-\frac12\Re\int_{\R} \Big(f(P_b+\eps)-f(P_b)\Big) \partial^2_y \phi_A\bar\eps \, dy
	+\Re\int_{\R} \partial^2_y\phi_A \Big(F(P_b+\eps)-F(P_b)-dF(P_b)\eps\Big)\, 
	dy\\+
	\Re\int_{\R} \Big(f(P_b+\eps)-f(P_b)-df(P_b)\eps\Big)
	\partial_y\phi_A\partial_y \bar P_b\, dy.
	\end{multline*}

Therefore, by combining Lemma \ref{lem:dsH} and Lemma \ref{lem:dsJ}, we have 
\begin{align*}
  \frac{d}{d s} \Big( H(s,\eps(s))\Big)
+ & b\cdot \frac{d}{d s}\Big(J(\eps(s))\Big)  \\
=& -
b\norm{\partial_y\eps}_2^2+ b\int_{\R}   \partial^2_y\phi_A|\partial_y\eps|^2 \, dy
\\
& +
\left(\frac{\lambda_s}{\lambda}+b\right) \cdot \Bigg(\norm{\partial_y \eps}_2^2
-\frac 12\, \Re\int_{\R} \Big(f(P_b+\eps)-f(P_b)\Big) \bar\eps \, dy\\
& \qquad \qquad \qquad 
+ \Re\int_{\R} \Big(F(P_b+\eps)-F(P_b)-dF(P_b)\eps\Big)\, dy
\\
& \qquad \qquad  \qquad  +\Re\int_{\R} \Big(f(P_b+\eps)-f(P_b)-df(P_b)\eps\Big)y\partial_y \bar P_b\,  dy\Bigg)
\\
&\qquad  +b\cdot \Bigg(-\frac12\, \Re\int_{\R} \Big(f(P_b+\eps)-f(P_b)\Big) \left(\partial^2_y \phi_A-1\right)\bar\eps \, dy\\
& \qquad  \qquad \quad +
\Re\int_{\R} \Big(F(P_b+\eps)-F(P_b)-dF(P_b)\eps\Big)\left(\partial^2_y \phi_A-1\right)\, dy\\
&\qquad  \qquad \quad  +\Re\int_{\R} \Big(f(P_b+\eps)-f(P_b)-df(P_b)\eps\Big)\left(\partial_y \phi_A-y\right)\partial_y \bar P_b\, dy\Bigg)
\\
& \qquad \qquad \qquad \qquad -\frac14\,  b \int_{\R}|\eps|^2\partial^4_y\phi_A\, dy +O(s^{-(2K+3)})+O(s^{-2}\norm{\eps}_{H^1}^2).
\end{align*}
By Lemma \ref{lem:modul est}, we have $$\left|\frac{\lambda_s}{\lambda}+b\right|\lesssim \left|{\rm Mod(s)}\right|\lesssim O(s^{-(K+2)}),$$
which together with \eqref{est:small est:rough} implies that  
\begin{align*}
\left|\left(\frac{\lambda_s}{\lambda}+b\right) \Bigg(\norm{\partial_y \eps}_2^2
-\frac 12\Re\int_{\R} \Big(f(P_b+\eps)-f(P_b)\Big) \bar\eps\, dy
+ \Re\int_{\R} \Big(F(P_b+\eps)-F(P_b)-dF(P_b)\eps\Big)\, dy\right.
\\ \left.
+\Re\int_{\R} \Big(f(P_b+\eps)-f(P_b)-df(P_b)\eps\Big)y\partial_y \bar P_b\, dy\Bigg)\right|\lesssim
s^{-(K+2)} \|\eps\|_{H^1}^{2}\lesssim s^{-(3K+2)}.
\end{align*}
Next, by \eqref{est:small est:rough} and the exponential decay of $P$, we have
\begin{multline*}
|b|  \left|-\frac12\Re\int_{\R} \Big(f(P_b+\eps)-f(P_b)\Big) \left(\partial^2_y\phi_A-1\right)\bar\eps\, dy\right|\\
\lesssim |b| \int_{\R}\Big(| P|^{4}|\partial^2_y\phi_A-1| |\eps|^2+|\eps|^{6} \Big)\, dy
\lesssim |b| \cdot e^{-\frac A2} \norm{\eps}_2^2+O\left(s^{-1}\norm{\eps}_{H^1}^{6}\right),
\end{multline*}
and similarly for 
$$b\cdot \Re\int_{\R} \Big(F(P_b+\eps)-F(P_b)-dF(P_b)\eps\Big)\left(\partial^2_y\phi_A-1\right)\,  dy,  $$ 
and 
$$b \cdot \Re\int_{\R} \Big(f(P_b+\eps)-f(P_b)-df(P_b)\eps\Big)\left(\partial_y \phi_A-y\right)\partial_y \bar P_b\, dy.$$
In addition, by the definition of $\phi_A$, we have
\begin{equation*}
\left|-b \int_{\R}|\eps|^2\partial^4_y\phi_Ady\right|\lesssim \frac{|b|}{A^2}\norm{\eps}_2^2.
\end{equation*}
Therefore, we can obtain \eqref{est:dsHJ}.

By Lemma \ref{lem:modul est}, we have $-\frac{\lambda_s}{\lambda} = b + O(s^{-(K+2)})$, thus we have 
\begin{align*}
&\quad -4\frac{\lambda_s}{\lambda} H(s,\eps) + \frac{d}{d s} \Big(H(s,\eps(s))\Big)
+b\cdot \frac{d}{d s}\Big(J(\eps(s))\Big)\\
&\gtrsim
4 b H(s,\eps) -b\norm{\partial_y \eps}_2^2 +b\int_{\R} \partial^2_y\phi_A|\partial_y\eps|^2 dy+O(s^{-2}\|\eps\|_{H^1}^2)+
\frac{b}{A}O(\norm{\eps}_{H^1}^2)+  O(s^{-(2K+3)})
\\&\gtrsim
 b\left( \int_{\R}\phantom{}  \partial^2_y\phi_A |\partial_y \eps|^2 \, dy+\norm{\eps}_2^2-2\int_{\R} \Big(F(P_b+\eps)-F(P_b)-dF(P_b)\eps\Big)\, dy\right)
\\
&\qquad +2 bH(s,\eps)+
O\left(\frac{b}{A}\norm{\eps}_{H^1}^2\right)+O(s^{-(2K+3)}).
\end{align*}
Thus, by choosing $A$ large enough, we can obtain from Lemma \ref{lem:H coer} and Lemma \ref{lem:locH:coer} that
\begin{align*}
-4\frac{\lambda_s}{\lambda} H(s,\eps) + \frac{d}{d s} \Big(H(s,\eps(s))\Big)
+b\cdot \frac{d}{d s}\Big(J(\eps(s))\Big)
\gtrsim b \|\eps\|_{H^1}^2 +O(s^{-(2K+3)}).
\end{align*}
Finally, by the facts that $b=O(s^{-1})$, $b_s=O(s^{-2})$ and $J(\eps)=O(\norm{\eps}_{H^1}^2)$, we have
$$\left( \left|\frac {\lambda_s}{\lambda}\right| b +|b_s|\right) |J(\eps)|\lesssim 
 s^{-2} \|\eps\|_{H^1}^2, $$
therefore, we have
\begin{equation*}
\frac{d}{d s}\Big(S(s,\eps(s))\Big) \gtrsim
\frac{b}{\lambda^{4}}\left(\norm{\eps}_{H^1}^2+O(s^{-(2K+2)})\right).
\end{equation*}
This completes the proof.
\end{proof}

\subsection{End of the proof of Proposition \ref{prop:unif est}}\label{sect:boot pf}

In this subsection, we finish the proof of Proposition \ref{prop:unif est}.
Recall from Subsection \ref{sect:boot base} that our goal is to prove $s_*=s_0$ by improving  \eqref{est:small est:rough} into \eqref{est:small est:s}. Therefore, it suffices to prove the following result.

 \begin{proposition}\label{prop:bootstrap}  
Let $K\geq 7$, then for all $s\in[s_*,s_1]$ , we have
\begin{align}
& \norm{\eps(s)}_{H^1}\lesssim s^{-(K+1)},\label{est:small eps}\\
&\left|\frac{\lambda^{\frac 12}(s)}{\lambda_{\rm app}^{\frac 12}(s)}-1\right|+
\left|\frac{b(s)}{b_{\rm app}(s)}-1\right|  
\lesssim s^{-1}.\label{est:small mod pars}
\end{align}
\end{proposition}

\begin{proof}
First, we prove \eqref{est:small eps}.
From  Proposition \ref{prop:S:coer},  there exists a constant $\kappa >1$ such that for any $s\in[s_*,s_1]$, we have
\begin{equation}\label{eq:bootstrap-1}
\frac{1}{\kappa}\frac 1{\lambda^{4}}\left(\norm{\eps}_{H^1}^2-\kappa^2 s^{-2(K+1)}\right)\leq  S(s,\eps) \leq \frac\kappa{\lambda^{4}} \norm{\eps}_{H^1}^2.
\end{equation}
By Proposition \ref{prop:dsS}, taking a larger $\kappa$ if necessary, we have
\begin{equation}\label{eq:bootstrap-2}
\frac{d}{d s}[S(s,\eps(s))]\geq \frac{1}{\kappa} \frac{b}{\lambda^{4}}\left(\norm{\eps}_{H^1}^2-\kappa^2s^{-2(K+1)}\right).
\end{equation}

Define 
\begin{equation*}
s_\dagger:=\inf\{ s\in[s_*,s_1],\quad \norm{\eps(\tau)}_{H^1}\leq   2\kappa^2\tau^{-(K+1)}\quad\text{for all }\tau\in[s,s_1] \}.
\end{equation*}
Since $\eps(s_1) \equiv 0$,  $s_\dagger$ is well-defined and  $s_\dagger<s_1$ by the continuity of the solution of \eqref{eq:dnls} in $H^1(\R)$.  We argue by contradiction.  Assume  that $s_\dagger>s_*$. In particular, we have $$\norm{\eps(s_\dagger)}_{H^1}=  2 \kappa^2 s_\dagger^{-(K+1)}.$$ 

Define 
\begin{equation*}
s_\ddagger:=\sup\{ s\in[s_\dagger,s_1],\quad \norm{\eps(\tau)}_{H^1}\geq  \kappa\tau^{-(K+1)}\quad\text{for all }\tau\in[ s_\dagger,s] \}.
\end{equation*}
In particular, we have 
$s_\dagger<s_\ddagger<s_1$ and $$\norm{\eps(s_\ddagger)}_{H^1}=\kappa s_\ddagger^{-(K+1)}.$$ From \eqref{eq:bootstrap-2}, $S$ is nondecreasing on $[s_\dagger,s_\ddagger]$.
By \eqref{est:small est:rough},   \eqref{eq:bootstrap-1} and \eqref{eq:bootstrap-2},  we have for $K\geq 7$ that
\begin{multline*}
\norm{\eps(s_\dagger)}_{H^1}^2-\kappa^2 s_\dagger^{-2(K+1)}
 \leq 
\kappa\lambda^{4}(s_\dagger)S(s_\dagger,\eps(s_\dagger))
 \leq
\kappa\lambda^{4}(s_\dagger)S(s_\ddagger,\eps(s_\ddagger))\\
\leq \kappa^2 \frac {\lambda^{4}(s_\dagger)}{\lambda^{4}(s_\ddagger)}\|\eps(s_\ddagger)\|_{H^1}^2
\leq
\kappa^4 \frac {\lambda^{4}(s_\dagger)}{\lambda^{4}(s_\ddagger)} s_\ddagger^{-2(K+1)}
\leq
2\kappa^4 \left(\frac{s_\ddagger}{s_\dagger}\right)^{8}s_\ddagger^{-2(K+1)}
\leq
2 \kappa^4  s_\dagger^{-2(K+1)},
\end{multline*} 
Therefore, we obtain $\norm{\eps(s_\dagger)}_{H^1}^2\leq 3 \kappa^4 s_\dagger^{-2(K+1)}$,
which is a contradiction with the definition of $s_\dagger$. Hence we obtain $s_\dagger=s_*$ and \eqref{est:small eps} is proved.

Now, we prove \eqref{est:small mod pars}.  Recall that
$\lambda(s_1)=\lambda_1$ and  $b(s_1)=b_1$ are chosen in Lemma \ref{lem:initial conds} so that
$\mathcal F(\lambda(s_1)) = s_1$ and $\displaystyle \mathcal{E}(b(s_1),\lambda(s_1))=\frac{8E_0}{\norm{yQ}^2_{L^2} }.$
In particular, we deduce from \eqref{est:Pblamb energy} that

\begin{equation*}
|E(P_{b_1,\lambda_1,\gamma_1})-E_0| \lesssim \frac{(b_1^2+\lambda_1)^{K+2}}{\lambda_1^2} \lesssim s_1^{-2K}.
\end{equation*}
By \eqref{est:Pblamb denergy}, \eqref{est:small est:rough}, and \eqref{est:modul}, we have for all $s\in[s_*,s_1]$ that
$$
\left|\frac{d}{ds}E(P_{b,\lambda,\gamma})\right| \lesssim \frac{1}{\lambda^2} \Big(\big|{\rm Mod(s)} \big|+ \big(b^2+ \lambda \big)^{K+2}\Big)
\lesssim s^{-K+2}.
$$ 
In particular, by integration over $[s, s_1]$, we obtain that
\begin{align}\label{est:Eblambda}
|E(P_{b,\lambda,\gamma}(s))-E_0|\lesssim  |E(P_{b,\lambda,\gamma}(s))-E(P_{b_1,\lambda_1,\gamma_1})| +|E(P_{b_1,\lambda_1,\gamma_1})- E_0| 
\lesssim s^{-K+3}.
\end{align}
By \eqref{est:Pblamb energy} and \eqref{est:small est:rough}, we have  for $K\geq 7$ that
\begin{equation*}
\left|\mathcal{E}(b(s),\lambda(s))-\frac{8E_0}{\displaystyle \int_{\R} |y|^2Q^2 \, dy }\right|\lesssim s^{-K+3}+ O\left( \frac{(b^2+\lambda)^{K+2}}{\lambda^2} \right)\lesssim s^{-4}.
\end{equation*}
By the expression \eqref{def:Eblamb} of $\mathcal E$ with $\displaystyle C_0=\frac{8E_0}{\int_{\R} |y|^2Q^2\, dy}$, we have
$$
\left|b^2-2\beta \lambda -C_0\lambda^2\right|\lesssim \lambda(b^2+\lambda)+ \frac{\lambda^2}{s^4}  \lesssim \frac{\lambda}{s^2},
$$
where the term $\lambda(b^2+\lambda)$ comes from the higher order terms in the definition of $\mathcal{E}(b(s),\lambda(s))$ in \eqref{def:Eblamb}. 
Since $b\approx \lambda^{\frac 12}$ by \eqref{est:small est:rough},  we have
\begin{equation}
\label{est:b}
\left|b-\sqrt{2\beta \lambda +C_0\lambda^2}\right|\lesssim \frac{\lambda^{\frac12}}{s^2},
\end{equation} which together with \eqref{def:F} and $\left|\frac{\lambda_s}{\lambda}+b\right|\lesssim s^{-(K+2)}$ implies that
\begin{equation}\label{est:dsF}
 \left| \frac{d}{d s} \Big(\mathcal F\big(\lambda(s)\big) \Big)- 1 \right|=\left|\frac{\lambda_s}{\lambda \cdot  \sqrt{2\beta \lambda +C_0\lambda^2}}+1\right|
\lesssim s^{-2}.
\end{equation}
By integrating the above estimate over $[s,s_1]$, we obtain
$$
\left|\mathcal F(\lambda(s_1))- \mathcal F(\lambda(s)) - (s_1-s)\right|\lesssim s^{-1},
$$
and thus, by the choice $\mathcal F(\lambda(s_1)) = s_1$ in Lemma \ref{lem:initial conds}, we obtain
$$\mathcal F(\lambda(s)) = s + O( s^{-1}).$$
Therefore, by \eqref{est:F} and the definition of $\lambda_{\rm app}(s)$ in \eqref{sol:ode app}, we have
\begin{equation*}
\lambda(s)^{-\frac12} + O(1) = \lambda_{\rm app}(s)^{-\frac12} + O(s^{-1})
\Longrightarrow
\left|\frac{\lambda^{\frac 1 2}(s)}{\lambda_{\rm app}^{\frac 12}(s)}-1\right| \lesssim  
s^{-1}.
\end{equation*}
 We  insert this estimate into \eqref{est:b}  to obtain 
 \begin{multline*}
 b(s)=\sqrt{2\beta \lambda(s) + C_0 \lambda(s)^2} + O(s^{-3})\\
 = \sqrt{2\beta \lambda_{\rm app}(s)  + 2\beta \big(\lambda(s)-\lambda_{\rm app}(s)\big)+ C_0 \lambda(s)^2} + O(s^{-3}) \\
 =  \sqrt{2\beta \lambda_{\rm app}(s) }  +\frac{  2\beta \big(\lambda(s)-\lambda_{\rm app}(s)\big)+ C_0 \lambda(s)^2  }{  (2\beta \lambda_{\rm app}(s))^{\frac12} }+ O(s^{-3} ), 
 \end{multline*}
 which together with the definition of $b_{\rm app}$ in Lemma \ref{lem:sol:ode app} implies that 
$$ b(s) = b_{\rm app}(s) + O\left(s^{-1}b_{\rm app}(s)\right). $$

Lastly,  we can obtain \eqref{est:Pblamb energy2} from \eqref{est:Eblambda}, and completes the proof.
\end{proof}

\section{Proof of Theorem \ref{thm:thresh2}}\label{Sect4: Exist}

This section is devoted to the proof of Theorem \ref{thm:thresh2} by a standard compactness argument. 

As shown in Remark \ref{rem:tapp}, we obtain the relation $t_{\rm app}(s)$ corresponding to the approximate scaling parameter $\lambda_{\rm app}$ in Lemma \ref{lem:sol:ode app}.  We now rewrite the estimates of Proposition \ref{prop:unif est} in the  time variable $t$ corresponding to the scaling parameter $\lambda$ in Proposition \ref{prop:unif est}.
 
\begin{proposition}\label{prop:unif est t}Let $t_1<0$ be close to $0$,  then there exists $t_0<0$ independent of $t_1$ such that under the assumptions of  Proposition \ref{prop:unif est}, for all $t\in [t_0,t_1]$, we have
	\begin{align}
	&b(t)= C_b |t|^{\frac 13}(1+o(1))
	, \ \ \lambda (t)=C_\lambda|t|^{\frac 23}(1+o(1)), \quad \text{as}\quad t\to 0^-,  \label{est:blamb t}\\
	&\norm{\eps(t)}_{H^1}\lesssim |t|^{\frac{K+1}{3}}, \label{est:eps t}\\
	& \label{est:Pblamb t}
\lim_{t\to 0^-}	E(P_{b,\lambda,\gamma}(t)) = E_0.
	\end{align}
\end{proposition}
\begin{proof} It suffices to show the relation $t(s) $ corresponding to $\lambda$ by Proposition \ref{prop:unif est}. By \eqref{def:st} and \eqref{est:small est:s},  we have  for large $s<s_1$ that
	$$
	t_1-t(s)=\int_s^{s_1}\lambda^2(\sigma)d\sigma=\int_s^{s_1}\lambda^2_{\rm app}(\sigma)\left[1+O(\sigma^{-1}) \right]d\sigma.$$
	Recall that $t_{\rm app}$ given by \eqref{def:tapp} corresponds to the normalization 
	$$t_{\rm app}(s)=-\int_{s}^{+\infty}\lambda^2_{\rm app}(\sigma)\, d\sigma,\quad t_{\rm app}(s_1)=t_1,$$
	from which we obtain
	$$t(s)=t_{\rm app}(s)(1+o(1))=-C_s s^{-3}\left[1+o(1)\right].$$ The estimates now follow directly follow from \eqref{sol:ode app} in Lemma \ref{lem:sol:ode app} and Proposition \ref{prop:unif est}  (see the definition of $C_\lambda$ and $C_b$ in \eqref{def:lapp} and \eqref{def:bapp}).
\end{proof}

Now, we can finish the proof of Theorem \ref{thm:thresh2} assuming Proposition \ref{prop:unif est} and Proposition \ref{prop:unif est t}.

\begin{proof}[Proof of Theorem \ref{thm:thresh2}]
	Let $(t_n)\subset(t_0,0)$ be an increasing sequence  such that 
	$$\lim_{n\to \infty}t_n=0.$$
	 For each $n$, let $u_n$  be the solution of \eqref{eq:dnls} on $[t_0,t_n]$ with final data at $t_n$ 
	\begin{equation}\label{eq:tn data}
	u_n(t_n,x)=\frac{1}{\lambda^{\frac{1}{2}}(t_n)} P_{b(t_n)}\left(\frac{x}{\lambda(t_n)}\right),
	\end{equation}
	where $\lambda(t_n)=\lambda_1$ and  $b(t_n)=b_1$ are given by  Lemma \ref{lem:initial conds}
	for $s_1=|C_s^{-1}t_n|^{-\frac{1}{3}}$, 
	so that $u_n(t)$ satisfies the conclusions of Proposition \ref{prop:unif est} and of Proposition \ref{prop:unif est t} on the interval $[t_0,t_n]$.
	The minimal mass blow up solution for \eqref{eq:dnls} is now obtained as the limit of a subsequence of $\{u_n\}_{n\geq 1}$. 
	
	In a first step, we prove that a subsequence  of $(u_n(t_0))$ converges to a suitable initial data. Indeed, from Proposition \ref{prop:unif est t}, we infer that $(u_n(t_0))$ is bounded in $H^1(\R)$.  Hence there exists a subsequence of $(u_n(t_0))$ (still denoted by $(u_n(t_0))$ and 
	$u_\infty(t_0)\in H^1(\R)$ such that 
	\[
	u_n(t_0)\rightharpoonup  u_\infty(t_0)\quad\text{weakly in }H^1(\R^d)\text{ as }n\to+\infty.
	\]
	Now, we obtain   strong  convergence in $H^s$ (for some $0<s<1$) by direct arguments.  Let $\chi:[0,+\infty)\to[0,1]$ be a smooth cut-off function such that $\chi\equiv 0$ on  $[0,1]$ and $\chi\equiv 1$ on $[2,+\infty)$. 
	For $R>0$, define $\chi_R:\R\to[0,1]$ by $\chi_R(x)=\chi(|x|/R)$. 
	Take any $\delta>0$. By the expression of $u_n(t_n)$ in \eqref{eq:tn data}, we can choose $R$ large enough (independent of $n$) so that  
	\begin{equation}\label{n-delta}
	\int_{\R} |u_n(t_n)|^2\chi_R dx\leq \delta.
	\end{equation}
	It follows from elementary computations that  
	\[
	\frac{d}{d t}\int_{\R} |u_n|^2\chi_R \, dx=2\, \Im\int_{\R}\nabla\chi_R\cdot \nabla u_n \, \bar u_n \, dx.
	\]
	Hence from the geometrical decomposition 
	\begin{equation*}
	u_n(t,x)=\frac{1}{\lambda_n^{\frac{1}{2}}(t)}\left(P_{b_n(t)}+\eps_n)(t,y)\right)e^{i\gamma_n(t)},\qquad 
	y=\frac{x}{\lambda_n(t)},
	\end{equation*}
	and the smallness \eqref{est:blamb t} and \eqref{est:eps t} of $\eps_n$  and $\lambda_n$ we infer
	\begin{equation*}
	\left|\frac{d}{d t}\int_{\R} |u_n(t)|^2\chi_R dx\right|\leq  \frac{C}{\lambda_n(t) R}\left(e^{-\frac{R}{2\lambda_n(t)}}+\norm{\eps_n(t)}^2_{H^1}\right)\leq \frac{C}{R}|t|^{-\frac 23}|t|^{\left(K+1\right)\frac 13}.
	\end{equation*}
	Integrating between $t_0$ and $t_n$, we obtain 
	\begin{equation*}
	\int_{\R} |u_n(t_0)|^2\chi_R dx\leq \frac{C}{R}|t_0|^{\left(-2+K+1\right)\frac 13+1}+\int_{\R}|u_n(t_n)|^2\chi_R dx.
	\end{equation*}
	Combined with \eqref{n-delta}, for a possibly larger $R$, this implies
	\begin{equation*}
	\int_{\R} |u_n(t_0)|^2\chi_R dx\leq 2\delta.
	\end{equation*}
	We conclude from the local compactness of Sobolev embedding that for $0\leq s<1$:
	\begin{equation*}
	u_n(t_0)\to  u_\infty(t_0)\quad\text{strongly in }H^s(\R),\ \text{ as }n\to+\infty.
	\end{equation*}
	Let $u_\infty(t)$ be the solution of \eqref{eq:dnls} with $u_\infty(t_0)$ as initial data at $t=t_0$. From \cite[Theorem 4.12.3]{Ca03:book}, there exists $0<s_0<1$ such that the Cauchy problem for \eqref{eq:dnls} is locally well-posed  in $H^{s_0}(\R)$. This implies that $u_\infty$ exists on $[t_0,0)$ and for any $t\in[t_0,0)$,  
	\begin{equation*}
	u_n(t)\to  u_\infty(t)\quad\text{strongly in }H^{s_0}(\R),\,\text{ weakly in }H^1(\R),\,\text{ as }n\to+\infty.
	\end{equation*}
	Moreover,
	since $\displaystyle \lim_{n\to \infty} \int_{\R} u_n^2(t_n) \, dx= \int_{\R} Q^2\, dx$, we have $\displaystyle \int_{\R} u_\infty^2 \, dx=\int_{\R} Q^2\, dx$. By weak convergence in $H^1(\R)$ and the estimates from Proposition \ref{prop:unif est t} applied to $u_n$, $u_\infty(t)$ satisfies \eqref{est:small tube}, and denoting $(\eps_\infty,\lambda_\infty,b_\infty,\gamma_\infty)$ its decomposition, we have by standard arguments (see e.g. \cite{MeR05:NLS:blp}), for any $t\in[t_0,0)$,
	$$
	\lambda_n(t)\to \lambda_\infty(t), \quad
	b_n(t)\to b_\infty(t), \quad
	\gamma_n(t)\to \gamma_\infty(t),
	\quad \eps_n(t)\rightharpoonup \eps_\infty(t) \quad \hbox{$H^1(\R)$ weak, as $n\to \infty$.}
	$$ 
	The uniform estimates on $u_n$ from Proposition \ref{prop:unif est t} give on $[t_0,0)$ that 
	\begin{equation}
	\label{est:small pars}
	b_\infty(t)=C_b|t|^{\frac 13} \left(1+o(1)\right) ,\quad \lambda_\infty (t)=C_\lambda |t|^{\frac 23} \left(1+o(1)\right) , \ \ \|\varepsilon_{\infty}(t)\|_{H^1}\lesssim |t|^{\frac{K+1}{3}},
	\end{equation}
	and 
	\begin{equation}\label{est:small pars2}
	\frac{b_\infty(t)}{\lambda_\infty^2(t)}=\frac{C_b}{C_\lambda^2}\, |t|^{\frac 13- \frac 43}
	\Big(1+o_{t\uparrow 0}(1)\Big)
	=\frac{2}{3}\, \frac{1}{|t|}\Big(1+o_{t\uparrow 0}(1)\Big),
	\end{equation}
	which justifies the form \eqref{sol:blp form} and the blow up rate \eqref{blp:rate2}.
	Finally, we prove that $E(u_\infty)=E_0$.
	Let $t_0<t<0$. We have by \eqref{est:Pblamb t} and \eqref{est:Pblamb energy},
	$$
	\lim_{t\to 0^-}\mathcal{E}(b_n(t),\lambda_n(t)) = \frac{8E_0}{\displaystyle \int_{\R} |y|^2 Q^2\, dy}
	$$
	where the limit is independent of $n$,
	and thus
	$$
		\lim_{t\to 0^-} \mathcal{E}(b_\infty(t),\lambda_\infty(t)) = \frac{8E_0}{\displaystyle \int_{\R} |y|^2 Q^2\, dy}.
	$$
	Using \eqref{est:Pblamb energy}, we deduce
	$$
		\lim_{t\to 0^-} E(P_{b_\infty,\lambda_\infty,\gamma_\infty}(t)) = E_0, $$
	and thus, by \eqref{est:small pars},
	$$
		\lim_{t\to 0^-} E(u_\infty(t))= E_0.
	$$
	Thus, by the energy conservation, we obtain $E(u_\infty(t))=E_0$.
\end{proof}

\begin{appendix}

\section{Proof of Theorem \ref{thm:thresh1}}
\label{app:a:gwp}
By contradiction, assume that there exists a radial blow-up solution $u(t)$ of \eqref{eq:dnls} with $\mu<0$ and $\|u(t)\|_{2}=\|Q\|_{2}$, which together with sharp Gagliardo-Nirenberg inequality and the energy conservation implies that $E_0=E(u(t))\geq E_{\rm crit}(u(t))\geq 0$. Let a sequence $t_n\to T^*\in (0,+\infty]$ with $\|\nabla u(t_n)\|_{2}\to +\infty$ and consider the renormalized sequence 
$$
v_n(x)=\lambda(t_n)^{\frac 12}u(t_n,\lambda (t_n)x), \ \ 0<\lambda(t_n)=\frac{\|\nabla Q\|_{2}}{\|\nabla u(t_n)\|_{2}} \searrow 0.
$$
Then, by the mass conservation, we have  $$\|v_n\|_{2}=\|Q\|_{2},$$ and by  the energy conservation and $\mu<0$, we have
$$
E_0=E(u(t_n))\geq E_{\rm crit}(u(t_n))=\frac{E_{\rm crit}(v_n)}{\lambda^2(t_n)}.
$$
Therefore, the sequence $v_n$ satisfies: 
$$
\|v_n\|_{2}=\|Q\|_{2}, \ \ \|\nabla v_n\|_{2}=\|\nabla Q\|_{2}, \ \ \limsup_{n\to +\infty} E_{\rm crit}(v_n)\leq 0.
$$ 
From standard concentration compactness argument, see \cite{MeR05:NLS:blp, We83:NLS:SGN}, there holds, up to a subsequence, for some  $\gamma_n\in \R$, 
$$
v_n(\cdot)e^{i\gamma_n}\longrightarrow Q\ \ \mbox{in} \ \ H^1(\R), \quad \text{as} \quad n\rightarrow +\infty.
$$ 
In particular,  we have as $n\to +\infty$ that
$$ 
|u(t_n, 0)|=\lambda^{-\frac12}(t_n) |v(t_n, 0)| = \lambda^{-\frac12}(t_n) | Q(0)+o_n(1)| \rightarrow +\infty,
$$ 
which contradicts with the a priori bound from the energy conservation
$$
E_0=E(u(t))\geq E_{\rm crit}(u(t))-\frac{1}{
2} \mu |u(t,0)|^2\geq -\frac{1}{
2} \mu |u(t,0)|^2.
$$
We conclude the proof.

\section{Proof of Proposition \ref{prop:smallsoliton}}
\label{app:b}
For $M<\|Q\|_{2}$,  we define
 $$A_M=\{u\in H^1_{\rm rad}(\R)\ \ \mbox{with}\ \ \|u\|_{2}=M\}$$  and consider the minimization problem $$I_M=\inf_{u\in A_M}E(u).$$ 
 
First, we claim
\begin{equation}
\label{signim}
-\infty<I_M<0.
\end{equation}
Indeed, from \eqref{est:energy} and $M<\norm{Q}_2$, we have 
 $I_M>-\infty$  and that any minimizing sequence is bounded in $H^1_{\rm rad}(\R)$.
 Let $u\in A_M$ and $v_{\lambda}(x)=\lambda^{\frac 12}u(\lambda x),$ then $v_\lambda \in A_M$ and $$E(v_\lambda)=\lambda^2\left[E_{\rm crit}(u)-\frac 1 {2\lambda} \mu |u(0)|^2\right].$$ 
 In particular, for $0<\lambda \ll1 $ and $u\not \equiv0$, that $E(v_\lambda)<0$  follows from the symmetric decreasing rearrangement if necessary in \cite{LiebLoss:book}, and we obtain \eqref{signim}.

Second, let $u_\lambda=u(\lambda x)$, so that
\begin{equation*}
\norm{u_\lambda}_2=\lambda^{-\frac12}\norm{u}_2,
\end{equation*}
and  
\begin{equation*}
E(u_\lambda)= \frac 12\lambda \int_{\R}\nabla u|^2\, dx -\frac1{2} \mu |u(0)|^2-\frac{1}{6\lambda}\int_{\R} |u|^{6}\, dx.
\end{equation*}
Therefore, we have
\begin{equation*}
\frac{d}{d\lambda}{\|u_\lambda\|_{2}}\Big|_{\lambda=1}<0,
\end{equation*}
and
\begin{equation*}
\frac{d}{d\l}E(u_\lambda)\Big|_{\lambda=1}= \frac 12\int_{\R}|\nabla u|^2\, dx +\frac 16 \int_{\R} |u|^{6}\, dx, 
\end{equation*}
which implies that
\begin{equation}
\label{monotonie}
I(M)\text{ is decreasing in }M.
\end{equation}

To finish, let $(u_n)$ be a minimizing sequence. Up to a subsequence and from the standard radial nonincreasing compactness of Sobolev embedding (see \cite{BL83:ell PDE}), we have
$$u_n\rightharpoonup u\ \ \mbox{in}\ \ H^1(\R), \ \ u_n\to u\ \ \mbox{in}\ \ L^6, \quad  \text{and}\quad u_n(0) \to u(0)$$ 
as $n\to +\infty$. Hence $$ E(u)\leq I_M\quad \hbox{and} \quad \|u\|_{2}\leq M.$$ From \eqref{monotonie} and the definition of $I_M$, we deduce $\|u\|_{2}=M$ and $E(u)=I_M$. From a standard Lagrange multiplier argument, $u$ satisfies $$\partial^2_x u+|u|^{5}u+\mu \delta u=\omega u$$
for a constant $\omega\in \R$. The fact that $\omega> \mu^2/4$  follows from a standard  spectral property of the Schr\"odinger operator $-\partial^2_x + \mu \delta$ in \cite{AlGKH88:Phys:book} .
\end{appendix}


\def\cprime{$'$}

\end{document}